\documentclass[11pt,reqno]{amsart}

\usepackage[top=1.15in, bottom=1.15in, left=1.17in, right=1.17in]{geometry}

\usepackage{amsmath, amssymb, amsthm,enumitem,bbm,bm,pict2e,xcolor}
\usepackage{indentfirst}

\theoremstyle{definition}
\newtheorem{theorem}{Theorem}
\newtheorem*{theorem*}{Theorem}
\numberwithin{theorem}{section}
\newtheorem{proposition}[theorem]{Proposition}
\newtheorem{lemma}[theorem]{Lemma}
\newtheorem{remark}[theorem]{Remark}

\newtheorem{cor}[theorem]{Corollary}

\DeclareMathOperator{\grad}{grad}
\DeclareMathOperator{\spanned}{span}

\DeclareMathOperator{\gr}{gr}
\DeclareMathOperator{\initial}{in}
\DeclareMathOperator{\trop}{trop}

\newcommand{\la}{\lambda}
\newcommand{\ra}{\rightarrow}
\newcommand{\om}{\omega}
\newcommand{\mU}{\mathcal U}

\newcommand{\bR}{\mathbb{R}}
\newcommand{\bs}{\bf s}

\title{Weighted PBW degenerations and tropical flag varieties}

\author{X. Fang}
\address{Xin Fang:\newline University of Cologne, Mathematical Institute, 
Weyertal 86--90, 50931, Cologne, Germany}
\email{xinfang.math@gmail.com}

\author{E. Feigin}
\address{Evgeny Feigin:\newline
Department of Mathematics, National Research University Higher School of Economics,
Usacheva str. 6, 119048, Moscow, Russia,\newline
{\it and }\newline
Skolkovo Institute of Science and Technology, Skolkovo Innovation Center, Building 3,
Moscow 143026, Russia
}
\email{evgfeig@gmail.com}

\author{G. Fourier}
\address{Ghislain Fourier:\newline Leibniz Universit\"at Hannover, Institute for Algebra, Number Theory
and Discrete Mathematics, Welfengarten 1, 30167 Hannover, Germany}
\email{fourier@math.uni-hannover.de}

\author{I. Makhlin}
\address{Igor Makhlin:\newline
Skolkovo Institute of Science and Technology, Skolkovo Innovation Center, Building 3,
Moscow 143026, Russia
\newline
{\it and }\newline
National Research University Higher School of Economics, 
International Laboratory of Representation Theory and Mathematical Physics,
Usacheva str. 6, 119048, Moscow, Russia}
\email{imakhlin@mail.ru}
\begin{document}
\begin{abstract}
We study algebraic, combinatorial and geometric aspects of weighted PBW-type degenerations of (partial) flag varieties in type $A$. These degenerations are labeled by degree functions lying in an explicitly defined polyhedral cone, which can be identified with a maximal cone in the tropical flag variety. Varying the degree function in the cone, we recover, for example, the classical flag variety, its abelian PBW degeneration, some of its linear degenerations and a particular toric degeneration. 
\end{abstract}

\maketitle

\section*{Introduction}
PBW degenerations of modules and projective varieties gained a lot of attention in the past decade, this fast growing subject provides new links between combinatorics, geometric representation theory, toric geometry and quiver Grassmannians to name but a few. The origin is a simple observation, namely let $\mathfrak n^-$ be the complex Lie algebra of strictly lower triangular $n\times n$-matrices and $U(\mathfrak n^-)$ be its universal enveloping algebra. 
By setting to one the degree of $f_{i,j}$ in the basis of elementary matrices, one obtains a filtration on $U(\mathfrak n^-)$ and every cyclically generated $\mathfrak n^-$-module (\cite{FFL1}). The associated graded structure is then abelian. This construction can be further transferred to (partial) flag varieties $F$, identified with highest weight orbits of an algebraic group, and provides the PBW degenerate flag variety \cite{Fe1}. This machinery was generalized in various directions, for an overview one may refer to \cite{CFFFR} or \cite{FFoL2}.

In this paper, we take a different approach. 
Instead of giving each $f_{i,j}$ degree $1$, we consider a weight system $A$, which is defined as a collection of integers $a_{i,j}$. Such a weight system induces a filtration of $\mathfrak n^-$  and $U(\mathfrak n^-)$. One could then ask for conditions on the weight system such that the associated graded vector spaces admit reasonable algebraic structures, e.g. the associated graded space of $\mathfrak n^-$ inherits a natural non-trivial graded Lie algebra structure.

In fact we provide an explicit description of a polyhedral cone $\mathcal{K}$, such that for any weight system $A$ in the cone one can use the machinery of PBW degenerations to construct degenerate cyclic modules and degenerate flag varieties $F^A$.

We study their combinatorial, algebraic and geometric properties, depending on the relative position of $A$ in $\mathcal{K}$. We point out that for certain weight systems discussions of the varieties $F^A$ can be found in the literature. For example, if $a_{i,j} = 0$ for all $i,j$, then we obtain just the classical flag variety; if $a_{i,j} = 1$ we obtain the PBW degenerate flag variety $F^a$ (\cite{Fe1}); if $a_{i,j} = (j-i+1)(n-j)$ we obtain the degeneration into a toric variety discussed in (\cite{FFR}, \cite{FFL3}). In fact, we prove that the latter degeneration is obtained for any weight system $A$ in the relative interior of $\mathcal{K}$ (Theorem~\ref{toricthm}). We also obtain (Remark~\ref{pbwlocus}) some of the linear degenerate flag varieties (PBW locus) discussed in (\cite{CFFFR}).

From now on, let $A$ be a weight system in $\mathcal{K}$, we show that one can embed (similarly to the classical case) the degenerate flag variety $F^A$ into a product of projective spaces of (degenerate) representation spaces. Classically, the ideal $I^0$ of Pl\"ucker relations on the Pl\"ucker coordinates describes the image of the flag variety in this product. Using our weight system, we naturally attach a degree $s_{\bullet}^A$ to each Pl\"ucker coordinate. Let $I^A$ be the initial ideal of $I^0$ with respect to these degrees. Then the first theorem is
\begin{theorem*}[Theorem \ref{main} and Proposition \ref{Aquadratic}]
The ideal $I^A$ is the defining ideal of $F^A$ with respect to the embedding above. Moreover, $I^A$ is generated by its quadratic part.
\end{theorem*}

In \cite{Fe1} a monomial basis in the homogeneous coordinate ring of the abelian degenerate flag variety $F^a$ has been constructed using PBW semi-standard Young tabelaux. It turns out that theses monomials form a basis in the homogeneous coordinate ring of $F^A$ for all $A \in \mathcal{K}$ (again generalizing results from \cite{FFL3}).

As we mentioned above, the degenerate flag varieties $F^A$ are labeled by weight systems belonging to a certain explicitly
described cone $\mathcal K$. $F^A$ only depends on the relative position in the cone and the flag varieties degenerate along the face lattice of $\mathcal{K}$:
\begin{theorem*}[Proposition \ref{facesame} and Proposition \ref{ABinitial}]
Let $H_A$ (resp. $H_B$) be the minimal face of $\mathcal{K}$ that contains a weight system $A$ (resp. $B$). If $H_A = H_B$, then $F^A \simeq F^B$ as projective varieties. Moreover, if $H_B \subseteq H_A$, then $F^A$ is a degeneration of $F^B$.
\end{theorem*}

Recall the degrees $s^A_{\bullet}$ attached to the Pl\"ucker coordinates. Let $\mathcal C=\{(s_{\bullet}^A) \mid A\in\mathcal K\}$ be the set of all 
collections of degrees obtained from all weight systems $A \in \mathcal{K}$. By construction, $F^A$ is irreducible and thus the initial ideal $I^A$ does not contain any monomials. This implies that $\mathcal C$ is contained in the tropical flag variety \cite{BLMM}. 
We prove the following theorem.
\begin{theorem*}[Theorem \ref{Trop}]
$\mathcal C$ is a maximal cone in the tropical flag variety.
\end{theorem*}
We derive explicit inequalities providing a non-redundant description of the facets of the cone $\mathcal C$.
To the best of our knowledge, this is the first appearance of a precise description of a maximal cone for each $n > 1$ (see \cite{SS,MaS,BLMM} for partial results in this direction).

The paper is organized as follows. In Section \ref{prelim} we recall definitions and results concerning the classical and PBW degenerate 
representations and flag varieties. In Section \ref{wPBW} we define the weighted PBW degenerations and derive  their 
first properties. In Section \ref{dPr} we write down the quadratic Pl\"ucker relations for the flag varieties $F_\la^A$.
Section \ref{proof} contains the proof of Theorem \ref{main} describing the ideals of the degenerate flag varieties. 
The toric degenerations corresponding to the weight systems 
in the interior of the cone of weight systems are considered in Section \ref{toric}. Section \ref{cone} discusses the cone
$\mathcal K$. The link to the theory of tropical flag varieties is described in Section \ref{tropical}. Finally, in 
Section \ref{BW} we prove a Borel-Weil-type theorem for the degenerate flag varieties $F_\la^A$.

\section{Preliminaries}\label{prelim}

\subsection{The classical theory}\label{classical}
For a Lie algebra $\mathfrak{g}$, let $\mathcal{U}(\mathfrak{g})$ denote the enveloping algebra associated with it.
Fix an integer $n\ge 2$ and consider the Lie algebra $\mathfrak{g}=\mathfrak{sl}_n(\mathbb C)$ of complex $n\times n$-matrices with trace 0. We have the Cartan decomposition $\mathfrak{g=n_+\oplus h\oplus n_-}$, the summands being, respectively, the subalgebra of upper triangular nilpotent matrices, the subalgebra of diagonal matrices and the subalgebra of lower triangular nilpotent matrices. Denote the corresponding set of simple roots $\alpha_1,\ldots,\alpha_{n-1}\in\mathfrak h^*$ and denote the set of positive roots $\Phi^+$. Then 
$$\Phi^+=\{\alpha_{i,j}=\alpha_i+\ldots+\alpha_{j-1}\mid 1\le i<j\le n\}.$$ For every positive root $\alpha_{i,j}$ we fix a non-zero 
element $f_{i,j}\in\mathfrak n_-$ in the weight space of weight $-\alpha_{i,j}$ in such a way that the following relation holds 
whenever $i\le k$: 
\begin{equation}\label{comm}
[f_{i,j},f_{k,l}]=
\begin{cases}
f_{i,l},\text{ if }j=k,\\
0, \text{ otherwise}. 
\end{cases}
\end{equation}
Let $\omega_1,\ldots,\omega_{n-1}$ be the fundamental weights of $\mathfrak{g}$.
Let $\la=a_1\om_1+\ldots+a_{n-1}\om_{n-1}\in\mathfrak{h}^*$ be an integral dominant weight for some $a_i\in\mathbb{Z}_{\geq 0}$. 
We denote $\bf d$ the tuple $(d_1,\ldots,d_s)$ where $\{d_1<\ldots<d_s\}$ is the set of all $i$ for which $a_i\neq 0$. 

Let $L_\la$ denote the finite-dimensional irreducible $\mathfrak g$-representation with highest weight $\la$. 
We fix a non-zero highest weight vector $v_\lambda\in L_\lambda$, then $L_\lambda=\mathcal{U}(\mathfrak{n}_-)v_\lambda$. 
Recall that for $h\in\mathfrak{h}$, $hv_\lambda=\lambda(h)v_\lambda$ and $\mathfrak{n}_+v_\lambda=0$.

Let $N$ be the Lie group of lower unitriangular complex $n\times n$-matrices, $N$ is diffeomorphic to $\mathbb C^{n \choose 2}$ 
and thus connected and simply connected. The 
Lie algebra $\mathrm{Lie}N$ can be naturally identified with $\mathfrak n_-$ 
which provides an action of $N$ on $L_\la$ via the exponential map. 
We then have $\mathbb C Nv_\la=L_\la$. Furthermore, we also obtain an action of $N$ 
on the projectivization $\mathbb P(L_\la)$. 
Let $u_\lambda\in \mathbb{P}(L_\la)$ be the point corresponding to the line $\mathbb{C}v_\lambda$.
The closure $\overline{Nu_\la}\subset\mathbb P(L_\la)$ is known as the (partial) flag variety which we will denote $F_\la$.

Let $V=\mathbb{C}^n$ be the vector representation of $\mathfrak{g}$ with basis $e_1,\ldots,e_n$ and $f_{i,j}$ mapping 
$e_i$ to $e_j$ while mapping $e_l$ with $l\neq i$ to 0. 
Then for $1\le k\le n-1$ we have $L_{\omega_k}=\wedge^k V$. For $1\leq i_1,\ldots,i_k\leq n$, we denote 
$$e_{i_1,\ldots,i_k}:=e_{i_1}\wedge\ldots\wedge e_{i_k}.$$
Then for any $\sigma\in S_k$, $$e_{i_1,\ldots,i_k}=(-1)^{\ell(\sigma)} e_{i_{\sigma(1)},\ldots,i_{\sigma(k)}}$$ where $\ell(\sigma)$ 
is the inversion number of $\sigma$. The vector space $L_{\omega_k}$ admits the basis 
$$\{e_{i_1,\ldots,i_k}\mid 1\le i_1<\ldots<i_k\le n\}.$$ The vector $e_{1,\ldots,k}$ is a highest weight 
vector and we assume that $v_{\om_k}=e_{1,\ldots,k}$. Note that $f_{i,j}$ maps $e_{i_1,\ldots,i_k}$ to 0 whenever 
$i\notin\{i_1,\ldots,i_k\}$ or $j\in\{i_1,\ldots,i_k\}$ and otherwise to $e_{i'_1,\ldots,i'_N}$ where $i'_l=i_l$ when 
$i_l\neq i$ and $i'_l=j$ when $i_l=i$.

We define the $\mathfrak g$-representation $$U_\la=L_{\om_1}^{\otimes a_1}\otimes\ldots\otimes L_{\om_{n-1}}^{\otimes a_{n-1}},$$ 
and denote 
$$w_\la=v_{\om_1}^{\otimes a_1}\otimes\ldots\otimes v_{\om_{n-1}}^{\otimes a_{n-1}}\in U_\la.$$ The subrepresentation 
$\mU(\mathfrak g)w_\la\subset U_\la$ is isomorphic to $L_\la$ by identifying $w_\la$ with $v_\la$.
We thus obtain the embedding $F_\la\subset\mathbb P(L_\la)\subset\mathbb P(U_\la)$.

We consider the Segre embedding $$\mathbb P(L_{\om_1})^{a_1}\times\ldots\times\mathbb P(L_{\om_{n-1}})^{a_{n-1}}\subset\mathbb P(U_\la)$$ 
and the embedding 
$$\mathbb P_{\bf d}=\mathbb P(L_{\om_{d_1}})\times\ldots\times\mathbb P(L_{\om_{d_s}})\subset \mathbb P(L_{\om_1})^{a_1}\times\ldots\times\mathbb P(L_{\om_{n-1}})^{a_{n-1}}$$ where $\mathbb P(L_{\om_{d_i}})$ is embedded diagonally into $\mathbb P(L_{\om_{d_i}})^{a_{d_i}}$. Let $y_\la$ be the point in $\mathbb P(U_\la)$ corresponding to $\mathbb Cw_\la$. The definition of the Segre embedding 
implies that $N y_\la\subset \mathbb P_{\bf d}$ which gives us an embedding $F_\la\subset \mathbb P_{\bf d}$.

The polynomial ring $$R_{\bf d}=\mathbb C[\{X_{i_1,\ldots,i_{d_j}}\mid 1\le j\le s,1\le i_1<\ldots<i_{d_j}\}]$$ is the homogeneous 
coordinate ring of $\mathbb P_{\bf d}$ with homogeneous coordinate $X_{i_1,\ldots,i_{d_j}}$ dual to the basis vector 
$e_{i_1,\ldots,i_{d_j}}\in L_{\om_{d_j}}$. In particular, the ring $R_{\bf d}$ is graded by $\mathbb Z_{\ge 0}^s$ with the generator 
$X_{i_1,\ldots,i_{d_j}}$ having homogeneity degree $(0,\ldots,1,\ldots,0)$ with the 1 being the $j$th coordinate. We embed 
$\mathbb Z_{\ge 0}^s$ into $\mathfrak h^*$ and view these homogeneity degrees as integral dominant weights by setting 
$\deg X_{i_1,\ldots,i_{d_j}}=\om_{d_j}$. 

Let us denote $X_{i_1,\ldots,i_p}=(-1)^{\ell(\sigma)} X_{i_{\sigma(1)},\ldots,i_{\sigma(p)}}$ for any 
$p\in\{d_1,\ldots,d_s\}$, $\{i_1,\ldots,i_p\}\subset\{1,\ldots,n\}$ and $\sigma\in S_p$. 
We consider the ideal $I_{\bf d}$ of $R_{\bf d}$ generated by the following quadratic elements (known as Pl\"ucker relations). 
For $1\leq k\leq q$ and a pair of collections of pairwise distinct elements $1\leq i_1,\ldots,i_p\leq n$ and 
$1\leq j_1,\ldots j_q\leq n$ where $p,q\in\{d_1,\ldots,d_s\}$ with $p\geq q$, the corresponding Pl\"ucker relation is:
\begin{equation}\label{plucker}
X_{i_1,\ldots,i_p}X_{j_1,\ldots,j_q}-\sum_{\substack{\{r_1,\ldots,r_k\}\subset\\\{i_1,\ldots,i_p\}}} 
X_{i'_1,\ldots,i'_p}X_{r_1,\ldots,r_k,j_{k+1},\ldots,j_q}
\end{equation}
where $i'_l=j_m$ if $i_l=r_m$ for some $1\le m\le k$ and $i'_l=i_l$ otherwise.  Note that the Pl\"ucker relations~(\ref{plucker}) are all homogeneous elements of the ring. 
\begin{theorem}\label{classicalplucker}
The ideal of the subvariety $F_\la\subset \mathbb P_{\bf d}$ is precisely $I_{\bf d}$. 
\end{theorem}
In particular, $F_\la$ only depends on the set $\{d_1,\ldots,d_s\}$. We also mention that the dimension of the component of $R_{\bf d}/I_{\bf d}$ of homogeneity degree $\la$ is equal to $\dim L_\la$.

The theory presented in this subsection can, for instance, be found in~\cite{carter,fulton}.

\subsection{Abelian PBW degenerations}\label{abelian}

We give a brief overview of the theory of abelian PBW (Poincar\'e--Birkhoff--Witt) degenerations following \cite{FFL1} and~\cite{Fe1}. 

The universal enveloping algebra $\mU=\mU(\mathfrak n_-)$ is equipped with a $\mathbb Z_{\ge 0}$-filtration known as the PBW filtration. 
The $m$th  component of the filtration is defined as 
$$(\mU)_m^a=\spanned(\{f_{i_1,j_1}\ldots f_{i_k,j_k}\mid k\le m\}),$$ 
the linear span of all PBW monomials of PBW degree no greater than $m$. 
By the PBW theorem, the associated graded algebra $\gr\mU$ is isomorphic to the symmetric algebra $\mathcal S^*(\mathfrak n_-)$.

The $\mathbb Z_{\ge 0}$-filtration on $\mU$ induces a $\mathbb Z_{\ge 0}$-filtration on $L_\la$ via $(L_\la)^a_m=(\mU)^a_m v_\la$. 
The associated graded space $\gr L_\la$ is naturally a module over the associated graded algebra $\gr\mU=\mathcal S^*(\mathfrak n_-)$. 
This $\mathcal S^*(\mathfrak n_-)$-module is most commonly known as the {\it PBW degeneration} of $L_\la$ and is denoted
by $L_\la^a$ with the ``$a$'' standing for ``abelian'' in reference to the commutativity of $\mathcal S^*(\mathfrak n_-)$. 
We will refer to this object as the {\it abelian PBW degeneration} or simply the {\it abelian degeneration} to distinguish 
it among the more general objects studied in this paper. Note that the component $(L_\la)^a_0$ is precisely $\mathbb Cv_\la$ 
and the 0th homogeneous component of $L_\la^a$ is one-dimensional, let $v_\la^a$ be the image of $v_\la$ therein. 
It can be easily seen that $L_\la^a=\mathcal S^*(\mathfrak n_-)v_\la^a$.

Now observe that $\mathcal S^*(\mathfrak n_-)$ is the universal enveloping algebra $\mU(\mathfrak n_-^a)$ of the abelian Lie algebra 
$\mathfrak n_-^a$ on the vector space $\mathfrak n_-$. We denote $f_{i,j}^a\in \mathfrak n_-^a$ the image of $f_{i,j}$ under the canonical linear isomorphism $\mathfrak n_-\to \mathfrak n_-^a$. 

$L_\la^a$ is a representation of $\mathfrak n_-^a$ and, consequently, of the corresponding connected simply connected Lie group $N^a$. The group $N^a$ is just $\mathbb C^{n\choose2}$ with $\mathbb C$ viewed as Lie group under addition. Further, we have an action of $N^a$ on $\mathbb P(L_\la^a)$. We denote $u_\la^a$ the point in $\mathbb P(L_\la^a)$ corresponding to $\mathbb Cv_\la^a$ and consider the closure $\overline{N^a u_\la^a}\subset\mathbb P(L_\la^a)$. This subvariety is known as the {\it PBW degenerate flag variety} or the {\it abelian degeneration} of $F_\la$ and is denoted $F_\la^a$.

Similarly to the above classical situation we define the $\mathfrak n_-^a$-representation 
$$U_\la^a=(L_{\om_1}^a)^{\otimes a_1}\otimes\ldots\otimes (L_{\om_{n-1}}^a)^{\otimes a_{n-1}}$$ 
and the vector $$w_\la^a=(v_{\om_1}^a)^{\otimes a_1}\otimes\ldots\otimes (v_{\om_{n-1}}^a)^{\otimes a_{n-1}}\in U_\la^a.$$ The subrepresentation $\mU(\mathfrak n_-^a)w_\la^a\subset U_\la^a$ is isomorphic to $L_\la^a$ via identifying $w_\la^a$ with $v_\la^a$.

We define the subvariety 
$$\mathbb P_{\bf d}^a=\mathbb P(L_{\om_{d_1}}^a)\times\ldots\times\mathbb P(L_{\om_{d_s}}^a)\subset\mathbb P(U_\la^a)$$ 
via the Segre embedding and obtain the embedding $F_\la^a\subset \mathbb P_{\bf d}^a$.

For every $1\le k\le n-1$ and $i_1<\ldots<i_k$ we may consider the least $m$ such that $e_{i_1,\ldots,i_k}\in(L_{\om_k})_m$ and denote $e_{i_1,\ldots,i_k}^a\in L_{\om_k}^a$ the image of $e_{i_1,\ldots,i_k}$ in $(L_{\om_k})_m /(L_{\om_k})_{m-1}$ (with $(L_{\om_k})_{-1}=0$). We will also use the notation $s_{i_1,\ldots,i_k}^a=m$.

The vectors $e_{i_1,\ldots,i_k}^a$ comprise a basis in $L_{\om_k}^a$. This allows us to view the ring 
$$R_{\bf d}^a=\mathbb C[\{X_{i_1,\ldots,i_{d_j}}^a\mid 1\le j\le s,1\le i_1<\ldots<i_{d_j}\}]$$ as the homogeneous coordinate ring of 
$\mathbb P_{\bf d}^a$ where the homogeneous coordinate $X_{i_1,\ldots,i_{d_j}}^a$ is the dual basis element to $e_{i_1,\ldots,i_{d_j}}^a$. 
We introduce an additional $\mathbb Z_{\ge 0}$-grading on $R_{\bf d}^a$ by setting 
$\grad^a X_{i_1,\ldots,i_{d_j}}^a=s_{i_1,\ldots,i_{d_j}}^a$. Now consider the isomorphism 
$$\varphi:R_{\bf d}\rightarrow R_{\bf d}^a,\ \ X_{i_1,\ldots,i_{d_j}}\mapsto X_{i_1,\ldots,i_{d_j}}^a.$$ 
Let $I_{\bf d}^a$ denote the initial ideal $\initial_{\grad^a}(\varphi(I_{\bf d}))$, i.e. the ideal spanned by the elements obtained 
by considering an $X\in\varphi(I_{\bf d})$ and then taking the sum of monomials in $X$ with lowest grading $\grad^a$ 
(the initial part $\initial_{\grad^a}X$ of $X$).

\begin{theorem}[\cite{Fe1}]\label{abplucker}
The ideal of the subvariety $F_\la^a\subset \mathbb P_{\bf d}^a$ is precisely $I_{\bf d}^a$.
\end{theorem}

The following more explicit characterization of $I_{\bf d}^a$ is also given in~\cite{Fe1}.
\begin{proposition}[\cite{Fe1}]\label{abquadratic}
$I_{\bf d}^a$ is generated by its quadratic part, i.e. it is generated by the initial parts (with respect to $\grad^a$) of the Pl\"ucker relations~(\ref{plucker}).
\end{proposition}


\subsection{FFLV bases and FFLV polytopes}

In~\cite{FFL1} certain combinatorial monomial bases in $L_\lambda$ and $L_\la^a$ are constructed. 
We briefly recall the definitions of these bases.

First we define the set $\Pi_\lambda$ that parametrizes the elements in each basis. This set is comprised of certain arrays 
of integers each containing ${n \choose 2}$ elements. Each array $T$ consists of elements $T_{i,j}$ with $1\le i<j\le n$. 
We visualize $T$ as a number triangle in the following way:
\begin{center}
\begin{tabular}{ccccccc}
$T_{1,2}$ &&$ T_{2,3}$&& $\ldots$ && $T_{n-1,n}$\\
&$T_{1,3}$ &&$ \ldots$&& $T_{n-2,n}$ &\\
&&$\ldots$ &&$ \ldots$& &\\
&&&$T_{1,n}$ &&&
\end{tabular}
\end{center}
Thus a horizontal row contains all $T_{i,j}$ with a given difference $j-i$.

To specify when $T\in\Pi_\lambda$ the notion of a {\it Dyck path} is used. We understand a Dyck path to be a sequence of pairs of integers 
$((i_1,j_1),\ldots,(i_N,j_N))$ with $1\le i<j\le n$ such that $j_1-i_1=j_N-i_N=1$ (both lie in the top row) and $(i_{k+1},j_{k+1})$ 
is either $(i_k+1,j_k)$ or $(i_k,j_k+1)$ for any $1\le k\le N-1$ (either the upper-right or the bottom-right neighbor of $(i_k,j_k)$). 
The set $\Pi_\lambda$ consists of all arrays $T$ whose elements are nonnegative integers such that for any Dyck path 
$$d=((i_1,j_1),\ldots,(i_N,j_N))$$ one has $$T_{i_1,j_1}+\ldots+T_{i_N,j_N}\le a_{i_1}+a_{i_1+1}+\ldots+a_{i_N}.$$ 
We will denote the right hand side by $M(\lambda,d)$. For a number triangle $T$ and a Dyck path $d$ we will 
denote the left-hand side above via $S(T,d)$. We will refer to $T\in\Pi_\la$ as {\it FFLV patterns}.

\begin{theorem}[\cite{FFL1}]
The set $$\left\{\left(\prod_{i,j} (f_{i,j}^a)^{T_{i,j}}\right) v_\la^a\mid T\in\Pi_\la\right\}$$ constitutes a basis in $L_\la^a$.
\end{theorem}
Note that the order of the factors in the above product does not matter in view of the abelianity of $\mathfrak n_-^a$. The following is easily seen to follow.
\begin{cor}
Each set of the form 
$$\left\{\left(\prod_{i,j} f_{i,j}^{T_{i,j}}\right) v_\la\mid  T\in\Pi_\la\right\},$$ 
where the order of the factors is chosen arbitrarily for each $T$, constitutes a basis in $L_\la$.
\end{cor}

Next we define the {\it FFLV polytope} $Q_\lambda$. The polytope is contained in $\mathbb R^{n \choose 2}$ with coordinates 
enumerated by pairs $1\le i<j\le n$. A point $x=(x_{i,j})$ in this space is visualized as a number triangle in the same exact 
fashion as the arrays comprising $\Pi_\lambda$. We have $x\in Q_\lambda$ if and only if all $x_{i,j}\ge 0$ and for any Dyck path 
$d$ one has 
$$S(x,d)\le M(\lambda,d).$$ We see that $Q_\lambda$ is indeed a convex polytope and $\Pi_\lambda$ is precisely the set of integer 
points therein.

An obvious but important property of the FFLV polytopes is $Q_\la+Q_\mu=Q_{\la+\mu}$ (Minkowski sum) for any pair of integral dominant weights. A much less obvious property proved in~\cite{FFL1} is the following {\it Minkowski property}.
\begin{lemma}\label{minkowski}
For any integral dominant weights $\lambda$ and $\mu$, one has $\Pi_\lambda+\Pi_\mu=\Pi_{\lambda+\mu}$, 
where $+$ is the Minkowski sum of sets.
\end{lemma}  

\subsection{PBW Young tableaux}\label{tableaux}

With the weight $\la$ let us associate the integers $\la_i=a_i+\ldots+a_{n-1}$ for $1\le i\le n-1$. These $\la_i$ comprise a non-increasing sequence of nonnegative integers and thus define a Young diagram (always in English notation) which we will also denote $\la$. Note that $a_i$ is precisely the number of columns of height $i$ in $\la$.

Consider $Y$, a filling of $\la$ (Young tableau of shape $\la$), we will denote $Y_{i,j}$ its element in the $i$th row and $j$th column. We say that $Y$ is a {\it PBW Young tableau} (or ``PBW tableau'' for short) if all of its elements are integers from $[1,n]$ and for every $1\le j\le \la_1$ the following hold (here $\la'_j$ denotes the height of the $j$th column in $\la$).
\begin{enumerate}
\item $Y_{i,j}\neq Y_{k,j}$ for any $1\le i\neq k\le \la'_j$.
\item For any $1\le i\le\la'_j$ if $Y_{i,j}\le\la'_j$, then $Y_{i,j}=i$.
\item For any $1\le i\neq k\le \la'_j$ if $Y_{i,j}>Y_{k,j}>\la'_j$, then $i<k$.
\end{enumerate}

We say that $Y$ is a PBW {\it semistandard} Young tableau (or ``PBW SSYT'' for short) if, apart from the above three conditions, 
whenever $j>1$ we also have
\begin{enumerate}\setcounter{enumi}{3}
\item For every $1\le i\le \la'_j$ there exists a $i\le k\le \la'_{j-1}$ such that $Y_{k,j-1}\ge Y_{i,j}$.
\end{enumerate}
We denote $\mathcal Y_\la$ the set of all PBW SSYTs of shape $\la$.

Let $Z$ be a PBW tableau of shape $\om_{d_j}$ for some $1\le j\le s$, i.e. consisting of one column of height $d_j$. Let 
$i_1<\ldots<i_{d_j}$ be the elements of $Z$ reordered increasingly. We denote $X(Z)=X_{i_1,\ldots,i_{d_j}}$. Next, for a PBW tableau 
$Y$ of shape $\la$ we denote $$X(Y)=\prod_{Z:\text{ column of }Y}X(Z)\in R_{\bf d}.$$ We will denote $X^a(Y)$ the image of 
$X(Y)$ under the isomorphism $\varphi$.

If $\mu$ is a $\mathbb Z_{\ge 0}$-linear combination of $\omega_{d_1},\ldots,\omega_{d_j}$, then the monomials $X(Y)$ with $Y$ 
ranging over all PBW tableaux of shape $\mu$ comprise a basis in the component of $R_{\bf d}$ 
of homogeneity degree $\mu$. We consider the ring $R_{\bf d}^a$ to also be $\mathfrak h^*$-graded (the grading being induced by $\varphi$) 
and the previous statement, also holds with $^a$ appended where necessary.

\begin{theorem}[\cite{Fe1}]\label{abpbwssyt}
For $\mu$ as above, the images of the monomials $X^a(Y)$ with $Y$ ranging over $\mathcal Y_\mu$ comprise a basis in the component of $R_{\bf d}^a/I_{\bf d}^a$ of homogeneity degree $\mu$.
\end{theorem}
In particular, since the corresponding homogeneous components of a homogeneous ideal and of its initial part with respect to some grading have the same dimension, we see that $|\mathcal Y_\la|=\dim L_\la$.

We point out that the analogous statement for the non-degenerate case also holds. This is not found in~\cite{Fe1} but will follow from the more general Corollary~\ref{Apbwssyt}.
\begin{theorem}
For $\mu$ as above, the images of the monomials $X(Y)$ with $Y$ ranging over $\mathcal Y_\mu$ comprise a basis in the component of $R_{\bf d}/I_{\bf d}$ of homogeneity degree $\mu$.
\end{theorem}

\begin{remark}\label{lattice}
One may also give the definition of PBW SSYTs in terms of a certain partial order. Consider the poset $\mathcal P$ the elements of which are nonempty proper subsets of $[1,n]$ and the order relation is defined as follows. For $x,y\in\mathcal P$ we set $x\preceq y$ whenever $|x|\ge|y|$ and the (unique) PBW tableau with two columns such that the $x$ is the set of elements in its first column and $y$ is the set of elements in its second column is PBW semistandard, i.e. satisfies conditions (2)-(4) above. (One may easily verify that this is indeed an order relation.) The above theorem then asserts that the monomials which are standard with respect to this partial order map to a basis in the homogeneous coordinate ring of $F_\la^a$.

The theory in this subsection should be reminiscent of the classical description of the coordinate ring $R_{\bf d}/I_{\bf d}$ in terms or semistandard Young tableaux which also provide bases in the homogeneous components and correspond to standard monomials with respect to a certain partial order (see~\cite{fulton,MiS} for details).
\end{remark}

To complete this section we prove a combinatorial fact that is not found in~\cite{Fe1}. First we introduce the partial order 
$\le$ on the set of all pairs $(i,j)$ with $1\le i<j\le n$ with $(i,j)\le (i',j')$ whenever $i\le i'$ and $j\le j'$. 
For $1\leq k\leq n-1$ we consider the set $\Pi_{\omega_k}$: it is comprised of all $T$ such that for $1\leq i<j\leq n$, 
(i) $T_{i,j}\in\{0,1\}$; (ii) if $T_{i,j}=1$ then $i\leq k$ and $j\geq k+1$; (iii) the set of pairs $(i,j)$ with $T_{i,j}=1$ 
forms an antichain with respect to $\le$.

For a PBW tableau $Z$ of shape $\om_k$ let us denote $\tau(Z)\in\Pi_{\om_k}$ the FFLV pattern with 
$$\tau(Z)_{i,j}=\begin{cases}1, & \text{ if }Z_{i,1}=j\text{ and }j>i;\\ 0, & \text{ otherwise.}\end{cases}$$ 
The above description of $\Pi_{\om_k}$ shows that 
this gives us a bijection between $\mathcal Y_{\om_k}$ and $\Pi_{\om_k}$. In fact, one sees that the element of the FFLV basis 
corresponding to $T$ is $e_{Z_{1,1}}\wedge\ldots\wedge e_{Z_{k,1}}$. 

We now define $\tau$ on any PBW tableau $Y$ of shape $\la$ by denoting $Z_i$ the PBW tableau of shape $\om_{\la'_i}$ found in the $i$th column of $Y$ and setting $\tau(Y)=\sum_{i=1}^{\la_1} \tau(Z_i)$.
\begin{lemma}\label{fflvtab}
When restricted to the set $\mathcal Y_\la$ the map $\tau$ provides a bijection onto $\Pi_\la$.
\end{lemma}
\begin{proof}
The fact that for $Y\in\mathcal Y_\la$ one has $\tau(Y)\in\Pi_\la$ is immediate from 
Lemma \ref{minkowski}.

Let us describe the inverse of $\tau$. Consider some $T\in\Pi_\la$ and consider $\la'_1=d_s$, i.e. the largest $i$ with $a_i>0$. 
Let $T^1\in\Pi_{\om_{\la'_1}}$ be such that $T^1_{i,j}=1$ whenever $i\le \la'_1$, $j\ge \la'_1+1$ and $(i,j)$ is maximal among all 
$(i',j')$ with $T_{i',j'}>0$ with respect to $\le$. We then observe that $T-T^1\in\Pi_{\la-\om_{\la'_1}}$ and define $T^2$ for the pair 
$\la-\om_{\la'_1}$, $T-T^1$ in the same way as $T^1$ was defined for the pair $\la$, $T$. By iterating this procedure we obtain a 
decomposition of $T$ into the sum $T^1+\ldots+T^{\la_1}$ with $T^i\in\Pi_{\om_{\la'_i}}$. 

Denote $Z_i$ the unique element of $\mathcal Y_{\om_{\la'_i}}$ with $\tau(Z_i)=T^i$ and denote $\zeta(T)$ the PBW tableau of shape $\la$ having $Z_i$ as its $i$th column. By definition, for any $1\le l<\la_1$ and for any $T^{l+1}_{i,j}=1$ one either has $j<\la'_l+1$ or one has $T^l_{i',j'}=1$ for some $(i',j')\ge (i,j)$. One easily checks that this is equivalent to the two-column tableau with first column $Z_l$ and second column $Z_{l+1}$ being PBW semistandard which shows that $\zeta(T)$ is PBW semistandard. It is now straightforward to verify that $\tau$ and $\zeta$ are mutually inverse.
\end{proof}

\section{Weighted PBW degenerations}\label{wPBW}

A weight system $A=(a_{i,j})_{1\leq i<j\leq n}$ is a collection of integers $a_{i,j}$ such that
\begin{enumerate}[label=(\alph*)]
\item $a_{i,i+1}+a_{i+1,i+2}\ge a_{i,i+2}$ for $1\le i\le n-2$ and
\item $a_{i,j}+a_{i+1,j+1}\ge a_{i,j+1}+a_{i+1,j}$ for $1\le i<j-1\le n-2$. 
\end{enumerate}
The reasons for these requirements should become evident below. 

We immediately derive a larger set of inequalities.
\begin{proposition}\label{moreineqs}
We have
\begin{enumerate}[label=(\Alph*)]
\item $a_{i,j}+a_{j,k}\ge a_{i,k}$ for $1\le i<j<k\le n$ and
\item $a_{i,j}+a_{k,l}\ge a_{i,l}+a_{k,j}$ for $1\le i<k<j<l\le n$. 
\end{enumerate}
\end{proposition}
\begin{proof}
In (a) and (b) replace $i$ with $i'$ and $j$ with $j'$.
Then the inequality in (A) can be obtained as the sum of the inequality in (a) for $i'=j-1$ and the inequalities in (b) for all 
$i\le i'\le j-1$ and $j\le j'\le k-1$ other than $i'=j-1$, $j'=j$.

The inequality in (B) can be obtained as the sum of the inequalities in (b) for all $i\le i'\le k-1$ and $j\le j'\le l-1$.
\end{proof}

The weight system allows us to define a $\mathbb Z$-filtration of the Lie algebra $\mathfrak n_-$ by setting for 
$m\in\mathbb Z$ $$(\mathfrak n_-)_m=\spanned(\{f_{i,j}\mid a_{i,j}\le m\}).$$ Condition (A) above together with the commutation 
relations~(\ref{comm}) ensure that we indeed have a filtered Lie algebra, i.e. $[(\mathfrak n_-)_l,(\mathfrak n_-)_m]\subset (\mathfrak n_-)_{l+m}$. Let $\mathfrak{n}_-^A$ be the associated graded Lie algebra 
$$\mathfrak n_-^A=\bigoplus_{m\in\mathbb Z} (\mathfrak n_-)_m/(\mathfrak n_-)_{m-1}=\bigoplus_{m\in\mathbb Z} (\mathfrak n_-^A)_m.$$

The Lie algebra  $\mathfrak n_-^A$ is spanned by elements $f_{i,j}^A$ with $f_{i,j}^A$ being the image of $f_{i,j}\in(\mathfrak n_-)_{a_{i,j}}$ in $(\mathfrak n_-^A)_{a_{i,j}}$. The commutation relations in $\mathfrak n_-^A$ are then given by: for $i\leq k$,
\begin{equation}\label{Acomm}
[f_{i,j}^A,f_{k,l}^A]=
\begin{cases}
f_{i,l}^A,\text{ if }j=k\text{ and }a_{i,j}+a_{k,l}=a_{i,l};\\
0, \text{ otherwise.}
\end{cases}
\end{equation}

For a PBW monomial $M=f_{i_1,j_1}\ldots f_{i_N,j_N}\in\mU$ define its $A$-degree to be $\deg^A M=a_{i_1,j_1}+\ldots+a_{i_N,j_N}$. We obtain a $\mathbb Z$-filtration on $\mU$ with the $m$th component being $$\mU_m=\spanned(\{M=f_{i_1,j_1}\ldots f_{i_N,j_N}\mid \deg^A M\le m\}).$$ 
We obtain a filtered algebra structure on $\mathcal{U}$ and denote the associated graded algebra $\mU^A=\gr\mU$. 
We will denote the $m$th homogeneous component $\mU_m/\mU_{m-1}=\mU^A_m$.

For every pair $1\le i<j\le n$ we may consider the element $f_{i,j}^A\in\mU^A_{a_{i,j}}$ which is the image of $f_{i,j}\in\mU_{a_{i,j}}$ (the reason for the apparent conflict of notations will become clear in a moment). It is evident that these $f_{i,j}^A\in\mU^A$ satisfy the commutation relations~(\ref{Acomm}) with the commutator being induced by the multiplication in $\mU^A$. Furthermore, the PBW theorem implies that these elements generate the algebra $\mU^A$. The universal property of $\mU(\mathfrak n_-^A)$ now shows that we have a surjective homomorphism from $\mU(\mathfrak n_-^A)$ to $\mU^A$ mapping $f_{i,j}^A\in\mathfrak n_-^A\subset\mU(\mathfrak n_-^A)$ to $f_{i,j}^A\in\mU^A$.
\begin{proposition}
The above homomorphism is an isomorphism.
\end{proposition}
\begin{proof}
Consider an arbitrary linear ordering ``$<$'' of the $f_{i,j}\in\mU$. Let us show that for $m\in\mathbb Z$ the monomials 
$M=f_{i_1,j_1}\ldots f_{i_N,j_N}$ with $f_{i_1,j_1}\le\ldots\le f_{i_N,j_N}$ and $\deg^A M\le m$ span $\mU_m$. Indeed, consider an arbitrary monomial $L\in\mU$ with $\deg^A L\le m$. Now apply the commutation relations~(\ref{comm}) to reorder the $f_{i,j}$ in $L$ and all the appearing summands to express $L$ as linear combination of $M=f_{i_1,j_1}\ldots f_{i_N,j_N}$ with $f_{i_1,j_1}\le\ldots\le f_{i_N,j_N}$ (this is the standard PBW procedure). The inequalities in (A) ensure that all the monomials in the resulting expression are of degree no greater than $m$.

Next, consider the induced ordering of the $f_{i,j}^A\in\mU^A$ which we also denote ``$<$''. To prove the proposition 
it suffices to show that all products $f_{i_1,j_1}^A\ldots f_{i_N,j_N}^A\in\mU^A$ for which $f_{i_1,j_1}^A\le \ldots\le f_{i_N,j_N}^A$ 
are linearly independent in $\mU^A$. Indeed, consider $k$ such products $M_1^A,\ldots,M_k^A$ and suppose that they are linearly dependent. 
We may assume that all $M_i^A\in\mU^A_m$ for some $m$. Now consider the corresponding monomials $M_1,\ldots,M_k\in\mU$ where $M_i$ is 
obtained from $M_i^A$ by simply removing the $^A$ superscript from each $f_{i,j}^A$. We see that a linear combination of these 
$M_1,\ldots,M_k\in\mU$ lies in $\mU_{m-1}$ which contradicts the conclusion from the previous paragraph.
\end{proof}

Now let us consider the $\mathbb Z$-filtration of $L_\la$ given by $(L_\la)_m=\mU_m v_\la$. Denote the associated graded space 
$L_\la^A$ and its homogeneous components $(L_\la^A)_m$. Since for $l,m\in\mathbb Z$ we have $\mU_l(L_\la)_m\subset(L_\la)_{l+m}$, 
the space $L_\la^A$ is naturally a graded $\mU^A$-representation. This representation is the {\it weighted PBW degeneration} of $L_\la$ and will be often referred to as ``degenerate representation'' for brevity. 

Note that $v_\la\in (L_\la)_0$ and that $v_\la\notin (L_\la)_m$ for all $m<0$. Therefore, we may consider the ``highest weight vector'' 
$v_\la^A\in L_\la^A$ which is the image of $v_\la$ in $(L_\la^A)_0$. We then have $L_\la^A=\mU^A v_\la^A$ and $(L_\la^A)_m=\mU^A_m v_\la^A$.

Since $L_\la^A$ is a representation of $\mU^A$ and, consequently, of $\mathfrak n_-^A$, it is also a representation of the 
corresponding connected simply connected Lie group which we denote $N^A$. Then $N^A$ also acts on the projectivization 
$\mathbb P(L_\la^A)$. In this projectivization we may consider the point $u_\la^A$ corresponding to $\mathbb Cv_\la$ and the closure 
$F_\la^A=\overline{N^Au_\la^A}$. This subvariety will be referred to as the {\it weighted PBW degeneration} of the partial flag 
variety associated with $\la$ or, more compactly, the {\it degenerate flag variety}.

\begin{remark}
If $a_{i,j}=0$ for all $1\leq i<j\leq n$, we obtain the non-degenerate objects $\mathfrak n_-^A=\mathfrak n_-$,  $L_\la^A=L_\la$ and 
$F_\la^A=F_\la$. If $a_{i,j}=1$ for all $1\leq i<j\leq n$,  we obtain the abelian degenerations $\mathfrak n_-^a$, $L_\la^a$ and 
$F_\la^a$. These observations will be generalized via Proposition~\ref{facesame} (see the paragraph following the Proposition).
\end{remark}

\begin{remark}
We see that in order to define $\mathfrak n_-^A$, $L_\la^A$ and $F_\la^A$ we only make use of the inequalities in (A). 
However, the inequalities in (B) become crucial in the subsequent sections and, in particular, in the proofs of 
Lemma~\ref{tensor} and Theorem~\ref{main}.
\end{remark}

\section{Degenerate Pl\"ucker relations}\label{dPr}

For $1\le k\le n-1$ consider the fundamental representation $L_{\om_k}$ (and recall its properties specified in 
subsection~\ref{classical}). We have $(L_{\om_k})_m=0$ for $m\ll 0$, hence for each tuple $1\le i_1<\ldots<i_k\le n$ 
we may choose the least $m$ such that $e_{i_1,\ldots,i_k}\in (L_{\omega_k})_m$. Let $e_{i_1,\ldots,i_k}^A\in L_{\omega_k}$ 
be the image of $e_{i_1,\ldots,i_k}$ in $(L_{\omega_k}^A)_m$.
\begin{proposition}
The vectors $e_{i_1,\ldots,i_k}^A$ comprise a basis in $L_{\omega_k}^A$.
\end{proposition}
\begin{proof}
This follows directly from the fact that the image of $v_{\om_k}$ under the action of any monomial in the $f_{i,j}$ is equal to $\pm e_{i_1,\ldots,i_k}$ for some $1\le i_1<\ldots<i_k\le n$.
\end{proof}
For each tuple $1\le i_1<\ldots<i_k\le n$ let $s^A_{i_1,\ldots,i_k}$ be the integer such that 
$e^A_{i_1,\ldots,i_k}\in(L_{\om_k}^A)_{s^A_{i_1,\ldots,i_k}}$.

We now proceed to generalize the constructions in subsections~\ref{classical} and~\ref{abelian} to give an explicit description of the homogeneous coordinate ring of $F_\la^A$ (with respect to a certain projective embedding) in terms of generators and relations. 

We define the $\mU^A$-representation $$U_\la^A=(L_{\om_1}^A)^{\otimes a_1}\otimes\ldots\otimes(L_{\om_{n-1}}^A)^{\otimes a_{n-1}}.$$ In $U_\lambda^A$ we distinguish the vector $$w_\la^A=(v_{\om_1}^A)^{\otimes a_1}\otimes\ldots\otimes(v_{\om_{n-1}}^A)^{\otimes a_{n-1}}.$$ The following key fact holds.
\begin{lemma}\label{tensor}
There is an isomorphism between $L_\la^A$ and the cyclic subrepresentation $\mU^A w_\la^A\subset U_\la^A$ mapping $v_\la^A$ to $w_\la^A$.
\end{lemma}
The proof of this lemma will be given in the next section.

Now let us consider the Segre embedding $$\mathbb P(L_{\om_1}^A)^{a_1}\times\ldots\times\mathbb P(L_{\om_{n-1}}^A)^{a_{n-1}}\subset\mathbb P(U_\la^A)$$ and the embedding $$\mathbb P_{\bf d}^A=\mathbb P(L_{\om_{d_1}}^A)\times\ldots\times\mathbb P(L_{\om_{d_s}}^A)\subset \mathbb P(L_{\om_1}^A)^{a_1}\times\ldots\times\mathbb P(L_{\om_{n-1}}^A)^{a_{n-1}}$$ where $\mathbb P(L_{\om_{d_i}}^A)$ is embedded diagonally into $\mathbb P(L_{\om_{d_i}}^A)^{a_{d_i}}$.

Denote $y_\la^A$ the point in $\mathbb P(U_\la^A)$ corresponding to $\mathbb Cw_\la^A$ and recall the Lie group $N^A$.
\begin{proposition}
We have $N^A y_\la^A\subset \mathbb P_{\bf d}^A\subset \mathbb P(U_\la^A)$. 
\end{proposition}
\begin{proof}
This follows immediately from the definition of the Segre embedding.
\end{proof}

In view of Lemma~\ref{tensor}, the closure $\overline{N^A y_\la^A}$ is precisely $F_\la^A$ and we have, therefore, embedded $F_\la^A$ into $\mathbb P_{\bf d}^A$. Let us consider homogeneous coordinates $X_{i_1,\ldots,i_{d_j}}^A$ on $\mathbb P_{\bf d}^A$ for all $1\le j\le s$ 
and all $1\le i_1<\ldots<i_{d_j}\le n$ with $X_{i_1,\ldots,i_{d_j}}^A$ being dual to $e^A_{i_1,\ldots,i_k}$.

We introduce a grading on the ring $R_{\bf d}^A=\mathbb C[\{X_{i_1,\ldots,i_{d_j}}^A\}]$ by setting 
$\grad^A X_{i_1,\ldots,i_{d_j}}^A=s^A_{i_1,\ldots,i_{d_j}}$. 
Let $\varphi^A:R_{\bf d}\rightarrow R_{\bf d}^A$ be the isomorphism sending $X_{i_1,\ldots,i_{d_j}}$ to $X_{i_1,\ldots,i_{d_j}}^A$, 
and $I_{\bf d}^A$ be the initial ideal $\initial_{\grad^A}(\varphi^A(I_{\bf d}))$.
\begin{theorem}\label{main}
The ideal of the subvariety $F_\la^A\subset P_d^A$ is precisely $I_{\bf d}^A$.
\end{theorem}
We postpone the proof of this theorem till the next section.
Also in Section \ref{cone} we will prove the following
\begin{proposition}\label{Aquadratic}
The ideal $I_{\bf d}^A$ is generated by its quadratic part.
\end{proposition}
This means that $I_{\bf d}^A$ is generated by the initial parts of the relations~(\ref{plucker}).

An important property of $F_\la^A$ is that it provides a {\it flat degeneration} of $F_\la$ in the following sense. 
Consider the ring $\mathfrak R_{\bf d}=R^A_{\bf d}\otimes\mathbb C[t]$, it is the coordinate ring of the variety 
$\overline{\mathbb P}^A_{\bf d}=\mathbb P^A_{\bf d}\times\mathbb A^1$ (the affine line). Now consider some $\delta\in I_{\bf d}$. 
{Via $(\varphi^A)^{-1}$ we may view $\grad^A$ as a grading on $R_{\bf d}$.} 
Let $M$ be the minimal degree with respect to $\grad^A$ among the monomials appearing in $\delta$ with a nonzero coefficient. 
Let $\theta^A$ be the homomorphism from $R_{\bf d}$ to $\mathfrak R_{\bf d}$ sending $X_{i_1,\ldots,i_k}$ to 
$t^{s_{i_1,\ldots,i_k}^A} X^A_{i_1,\ldots,i_k}$. Let $\mathfrak I_{\bf d}^A\subset\mathfrak R_{\bf d}$ be the ideal 
generated by the expressions $t^{-M}\theta^A(\delta)$ ranging over all $\delta$ (in fact, one sees that $\mathfrak I_{\bf d}^A$ 
is generated by the expressions $t^{-M}\theta^A(\delta)$ already as a $\mathbb C[t]$-module).

Denote $\mathfrak F_{\bf d}^A\subset\mathfrak \overline{\mathbb P}^A_{\bf d}$ the subvariety given by the ideal 
$\mathfrak I_{\bf d}^A$. Let $\pi_2: \overline{\mathbb P}^A_{\bf d}\ra\mathbb{A}^1$ be the projection onto the second component.
One sees that $(\mathfrak R_{\bf d}/\mathfrak I_{\bf d}^A)/(t)$ is isomorphic to $R^A_{\bf d}/I^A_{\bf d}$, i.e. the fiber 
$\pi_2^{-1}(0)\cap\mathfrak F_{\bf d}^A$ is equal to $F^A_\la$, while 
$(\mathfrak R_{\bf d}/\mathfrak I_{\bf d}^A)/(t-c)$ for $c\neq 0$ is isomorphic to $R_{\bf d}/I_{\bf d}$, i.e. any fiber 
$\mathfrak F_{\bf d}^A\cap \pi_2^{-1}(c)$ with $c\neq 0$ is isomorphic to $F_\la$.

The flatness of the degeneration will follow from the somewhat stronger
\begin{proposition}\label{free}
The ring $\mathfrak R_{\bf d}/\mathfrak I_{\bf d}^A$ is free over $\mathbb C[t]$.
\end{proposition}
\begin{proof}
The ring $\mathfrak R_{\bf d}$ is also $\mathfrak h^*$-graded by $\deg X^A_{i_1,\ldots,i_k}=\om_k$ and $\deg t=0$. 
The ideal $\mathfrak I_{\bf d}^A$ is $\deg$-homogeneous. We choose a weight $\mu\in\mathbb Z\{\omega_{d_1},\ldots,\omega_{d_s}\}$ 
and prove that $\mathfrak R_{\bf d,\mu}/\mathfrak I^A_{\bf d,\mu}$ is free over $\mathbb C[t]$ where $\mathfrak R_{\bf d,\mu}$ and 
$\mathfrak I^A_{\bf d,\mu}$ are the components of homogeneity degree $\mu$ in the respective spaces.

Denote $I_{\bf d,\mu}$ the component of homogeneity degree $\mu$ in $I_{\bf d}$. We may choose a basis 
$\delta_1,\ldots,\delta_{D_\mu}$ in $I_{\bf d,\mu}$ such that the initial parts $\initial_{\grad^A}(\varphi^A(\delta_i))$ comprise a basis in the 
homogeneous component $I_{\bf d,\mu}^A$. Indeed, consider the subspace of $I_{\bf d,\mu}$ of such $\delta$ that 
$\grad^A(\initial_{\grad^A}(\varphi^A(\delta)))$ is as large as possible and choose a basis in this subspace. Then consider the subspace of 
$I_{\bf d,\mu}$ of such $\delta$ that $\grad^A(\initial_{\grad^A}(\varphi^A(\delta)))$ takes 
one of the two largest possible values and extended the previously chosen basis to a basis in this space. Increasing the subspace by adding one value of $\grad^A(\initial_{\grad^A}(\varphi^A(\delta)))$ at a time will result in a basis with the desired property.

We may extended the basis $\{\delta_i\}$ to a basis in the homogeneous component $R_{\bf d,\mu}$ by a set of monomials 
$M_1,\ldots,M_{\dim L_\mu}$ in such a way that the set $\{\initial_{\grad^A}(\varphi^A(\delta_i))\}\cup\{\varphi^A(M_i)\}$ is a basis in 
$R_{\bf d,\mu}^A$. Let $\mathfrak O$ be the subset in $\mathfrak R_{\bf d,\mu}$ comprised of the expressions 
$t^{-\grad^A(\initial_{\grad^A}(\varphi^A(\delta_i)))}\theta^A(\delta_i)$ and $\mathfrak M$ be the subset in $\mathfrak R_{\bf d,\mu}$ comprised of the monomials $\varphi^A(M_i)$. It is straightforward to check that $\mathfrak O$ generates $\mathfrak I^A_{\bf d,\mu}$ as a $\mathbb C[t]$-module and $\mathfrak O\cup\mathfrak M$ is a $\mathbb C[t]$-basis in $\mathfrak R_{\bf d,\mu}$. The proposition follows.
\end{proof}

One sees that the above proof applies in the general case: one could take any homogeneous ideal instead of $I_{\bf d}$ and any grading 
instead of $\grad^A$.

\section{Proof of Theorem~\ref{main}}\label{proof}

In this section we will prove Lemma~\ref{tensor} and then derive Theorem~\ref{main}.

We start off by giving the following explicit description of the integers $s^A_{i_1,\ldots,i_k}$. Choose $1\le i_1<\ldots<i_k\le n$ and let the integers $p_1<\ldots<p_l$ comprise the difference $\{1,\ldots,k\}\backslash\{i_1,\ldots,i_k\}$ while the integers $q_1>\ldots>q_l$ comprise the difference $\{i_1,\ldots,i_k\}\backslash\{1,\ldots,k\}$. 
\begin{proposition}\label{antichainmin}
In the above notations $$s^A_{i_1,\ldots,i_k}=a_{p_1,q_1}+\ldots+a_{p_l,q_l}.$$
\end{proposition}
\begin{proof}
Consider a monomial $M=f_{x_1,y_1}\ldots f_{x_N,y_N}\in\mU$ such that $M v_{\om_k}=\pm e_{i_1,\ldots,i_k}$ of 
minimal possible degree $\deg^A$. Then $s^A_{i_1,\ldots,i_k}=\deg^A M$.

Let us also assume that among all possible choices, $M$ is comprised of the least possible number of $f_{x,y}$, i.e. $N$ is as small as possible with the above properties holding.

First, let us show that we have $x_j\le k$ and $y_j\ge k+1$ for all $1\le j\le N$. Indeed, consider the largest $j$ such that either 
$x_j>k$ or $y_j<k+1$. Suppose first that $x_j>k$. Note that $f_{x_j,y_j}$ acts nontrivially only on those $e_{i'_1,\ldots,i'_k}$ for 
which $x_j\in\{i'_1,\ldots,i'_k\}$. Since $v_{\om_k}=e_{1,\ldots,k}$, this implies that there exists some $j'>j$ such that $y_{j'}=x_j$. 
Since all $f_{x,y}$ with $x\le k$ and $y\ge k+1$ commute pairwisely, we may assume that $j'=j+1$. Now, the image of 
$f_{x_{j+2},y_{j+2}}\ldots f_{x_N,y_N} v_0$ under the action of $f_{x_j,y_j}f_{x_{j+1},x_j}$ coincides with its image under the action 
of $f_{x_{j+1},y_j}$, i.e. we may replace $f_{x_j,y_j}f_{x_{j+1},y_{j+1}}$ with $f_{x_{j+1},y_j}$ and obtain a monomial 
$M'$ of no greater $\deg^A$ (due to inequality (A)) such that $M'v_0=Mv_0$. 

If $y_j<k+1$, then we may assume that $x_{j+1}=y_j$ and define $M'$ of no greater degree by replacing 
$f_{x_j,y_j}f_{x_{j+1},y_{j+1}}$ with $f_{x_j,y_{j+1}}$.

Now observe that a product of the form $f_{x,y}f_{x,y'}$ or $f_{x,y}f_{x',y}$ annihilates $L_{\om_k}$. Therefore, with the mentioned commutativity taken into account, all $x_j$ as well as all $y_j$ are pairwise distinct. This means that $N=l$ and the set $\{x_1,\ldots,x_N\}$ is precisely $\{p_1,\ldots,p_l\}$ while $\{y_1,\ldots,y_N\}$ is precisely $\{q_1,\ldots,q_l\}$. Now suppose that we have $x_j\le x_{j'}$ and $y_j\le y_{j'}$ for some $j\neq j'$. Then we may replace $f_{x_j,y_j}f_{x_{j'},y_{j'}}$ with $f_{x_j,y_{j'}}f_{x_{j'},y_j}$ and obtain a monomial that is of no greater degree (due to inequality (B)) and once again sends $v_{\om_k}$ to $\pm e_{i_1,\ldots,i_k}$. Repeating this procedure we will eventually obtain a monomial in which either $x_j> x_{j'}$ or $y_j> y_{j'}$ for any $j\neq j'$, i.e. precisely the monomial $f_{p_1,q_1}\ldots f_{p_l,q_l}$.
\end{proof}

The above proposition is directly related to FFLV bases and PBW tableaux. Indeed, the discussion preceding Lemma~\ref{fflvtab} implies 
that if $Z$ is the unique one-column PBW tableau with content $\{i_1,\ldots,i_k\}$, then $\tau(Z)_{p_j,q_j}=1$ for all $1\le j\le l$ 
while all other $\tau(Z)_{i,j}=0$. Therefore, Proposition~\ref{antichainmin} tells us that the basis $\{e^A_{i_1,\ldots,i_k}\}$ 
coincides (up to sign change) with $$\left\{\left(\prod (f_{i,j}^A)^{T_{i,j}}\right)v_{\om_k}^A\mid T\in\Pi_{\om_k}\right\}.$$ 
We will denote the FFLV pattern $\tau(Z)$ corresponding to the basis vector $e^A_{i_1,\ldots,i_k}$ via $T(e^A_{i_1,\ldots,i_k})$.

We next prove
\begin{proposition}\label{fflvlinind}
For every $T\in\Pi_\la$ choose a monomial $M_T\in\mU^A$ of the form $\prod (f_{i,j}^A)^{T_{i,j}}$ where the factors are ordered arbitrarily. The set $\{M_Tw_\la^A\mid T\in\Pi_\la\}$ is linearly independent. 
\end{proposition}
\begin{proof}
The space $U_\la^A$ has a basis comprised of all vectors of the form $e=e^1\otimes\ldots\otimes e^{a_1+\ldots+a_{n-1}}$ where the first $a_1$ factors are of the form $e^A_i\in L_{\om_1}^A$, the next $a_2$ are of the form $e^A_{i_1,i_2}\in L_{\om_2}^A$ and so on. With each such basis vector $e$ we may associate $T(e)\in\Pi_\la$. Indeed, set $T(e)=T(e^1)+\ldots+T(e^{a_1+\ldots+a_{n-1}})$ (a sum of points in $\mathbb R^{n\choose 2}$). We have $T(e)\in\Pi_\la$ in view of the Minkowski property Lemma~\ref{minkowski}. We have decomposed $U_\la^A$ into the direct sum of spaces $$(U_\la^A)_T=\bigoplus_{T(e)=T}\mathbb Ce$$ with $T$ ranging over $\Pi_\la$.

Next, let us define a partial order on the set of all number triangles $T=\{T_{i,j}\mid 1\le i<j\le n\}$ and, in particular, on $\Pi_\la$. For such a triangle $T$ let $\gamma(T)$ be the sequence of all elements $T_{i,j}$ ordered first by $i+j$ increasing and then by $i$ increasing, i.e. from left to right and within one vertical column from the bottom up. We then write $T_1\preceq T_2$ whenever $\gamma(T_1)$ is no greater than $\gamma(T_2)$ lexicographically, that is $T_1=T_2$ or for the least $i$ such that $\gamma(T_1)_i\neq\gamma(T_2)_i$ one has $\gamma(T_1)_i<\gamma(T_2)_i$. Note that the order $\preceq$ is additive: $T_1\preceq T_2$ and $T_3\preceq T_4$ imply $T_1+T_3\preceq T_2+T_4$.

This partial order induces a partial order on monomials in $\mU^A$ by setting $$\prod (f_{i,j}^A)^{T^1_{i,j}}\preceq \prod (f_{i,j}^A)^{T^2_{i,j}}$$ whenever $T^1\preceq T^2$ for arbitrary orderings of the factors. This order is multiplicative.

Consider some $1\le i_1<\ldots<i_k\le n$ and a monomial $M$ with $Mv^A_{\om_k}=\pm e_{i_1,\ldots,i_k}^A$. First, we must have $\deg^A M=s^A_{i_1,\ldots,i_k}$. Second, the proof of Proposition~\ref{antichainmin} shows that the monomial $\prod (f_{i,j}^A)^{T(e^A_{i_1,\ldots,i_k})_{i,j}}$ can be obtained from $M$ by replacing $f^A_{i,j}f^A_{j,l}$ with $f_{i,l}^A$ for $i<j<l$, replacing $f^A_{i,j}f^A_{l,m}$ with $f^A_{i,m}f^A_{l,j}$ for $i<l<j<m$ and commuting the factors. None of these operations increase the monomial with respect to $\preceq$, i.e. $M\succeq \prod (f_{i,j}^A)^{T(e^A_{i_1,\ldots,i_k})_{i,j}}$.

Now consider $T\in\Pi_\la$ and the vector $M_T w_\la^A$. This vector decomposes as 
\begin{multline}\label{tensordecomp}
M_T w_\la^A=\\\sum_{\substack{T^{1,1}+\ldots+T^{1,a_1}+\ldots+\\T^{n-1,1}+\ldots+T^{n-1,a_{n-1}}=T}}\bigg(\left(\prod (f_{i,j}^A)^{T^{1,1}_{i,j}}v_{\om_1}^A\right)\otimes\ldots\otimes\left(\prod (f_{i,j}^A)^{T^{1,a_1}_{i,j}}v_{\om_1}^A\right)\otimes\ldots\otimes\\\left(\prod (f_{i,j}^A)^{T^{n-1,1}_{i,j}}v_{\om_{n-1}}^A\right)\otimes\ldots\otimes\left(\prod (f_{i,j}^A)^{T^{n-1,a_{n-1}}_{i,j}}v_{\om_{n-1}}^A\right)\bigg).
\end{multline}
Here the sum ranges over all decompositions of $T$ into a sum of number triangles $T^{k,l}$ with non-negative integer elements over all $1\le k\le n-1$ and $1\le l\le a_k$. The set of all factors $f_{i,j}$ over all the products $\prod (f_{i,j}^A)^{T^{k,l}_{i,j}}$ is in one-to-one correspondence with the factors in $M_T$ and in each such product they are ordered in the same way as in $M_T$. Suppose 
$\prod (f_{i,j}^A)^{T^{k,l}_{i,j}}v_{\om_1}^A$ is nonzero and equal to $\pm e_{i_1,\ldots,i_k}^A$. This implies that 
$\deg^A(\prod (f_{i,j}^A)^{T^{k,l}_{i,j}})=s_{i_1,\ldots,i_k}^A$ and, in view of the discussion above, that 
$T^{k,l}\succeq T(e_{i_1,\ldots,i_k}^A)$. This means that every summand in the right-hand side of~(\ref{tensordecomp}) is contained in some 
$(U_\la^A)_{T'}$ with $T'\preceq T$. Moreover, the Minkowski property implies that for at least one summand we have $T^{k,l}\in\Pi_{\om_k}$ for all $k$ and $l$, i.e. at least one of the summands is contained in $(U_\la^A)_T$. The linear independence follows.
\end{proof}

We have thus shown that $\dim(\mU^Aw_\la^A)\ge|\Pi_\la|=\dim L_\la=\dim L_\la^A$. To prove Lemma~\ref{tensor} it now suffices to show the existence of a surjective  homomorphism from $L_\la^A$ to $\mU^Aw_\la^A$ taking $v_\la^A$ to $w_\la^A$.
\begin{proof}[Proof of Lemma~\ref{tensor}]
We are to show that for every $S\in\mU^A$ we have $S w_\la^A=0$ whenever $S v_\la^A=0$. Indeed, consider some $S$ with $S v_\la^A=0$, we may assume that $S\in \mU^A_m$ for some integer $m$.

The relation $S v_\la^A=0$ means that there exists some $S'\in\mU$ such that $S'v_\la=0$ and $S'=S_0+S_1$ with $S_0$ being obtained from $S$ by simply removing all the $^A$ superscripts and $S_1\in\mU_{m-1}$. We then have $S'w_\la=0\in U_\la$.

$U_\la$ has a basis comprised of all vectors of the form $e=e^1\otimes\ldots\otimes e^{a_1+\ldots+a_{n-1}}$ where the first $a_1$ factors are of the form $e_i\in L_{\om_1}$, the next $a_2$ are of the form $e_{i_1,i_2}\in L_{\om_2}$ and so on. For such an $e$ suppose that some $e^i$ is equal to $e_{i_1,\ldots,i_k}$ and let $m_i$ be the least integer such that $e_{i_1,\ldots,i_k}\in (L_{\om_k})_{m_i}$, i.e. $m_i=s^A_{i_1,\ldots,i_k}$. We then set $m^A(e)=m_1+\ldots+m_{a_1+\ldots+a_{n-1}}$ and obtain the decomposition of $U_\la$ into a direct sum of the spaces $$(U_\la)_N=\bigoplus_{m^A(e)=N}\mathbb Ce.$$

Now consider the map $\Psi$ from $U_\la$ to $U_\la^A$ obtained as the tensor product of maps from $L_{\om_k}$ to $L_{\om_k}^A$ taking $e_{i_1,\ldots,i_k}$ to $e_{i_1,\ldots,i_k}^A$. Let $M\in\mU$ be a monomial with $\deg^A M=m$ and $M^A\in\mU^A$ be the monomial obtained from $M$ by adding $^A$ superscripts to all the factors. By considering the decompositions of $Mw_\la$ and $M^Aw_\la^A$ in the respective bases, one sees that $Mw_\la\in\bigoplus_{m'\le m}(U_\la)_{m'}$ and that $M^Aw_\la^A=\Psi((Mw_\la)_m)$ with $(.)_m$ denoting the projection onto $(U_\la)_m$. In particular, $(S_1w_\la)_m=0$ and, therefore, $$S w_\la^A=\Psi((S_0w_\la)_m)=\Psi((S'w_\la)_m)=0.$$ 
\end{proof}

Together Proposition~\ref{fflvlinind} and Lemma~\ref{tensor} have the following implication, that connects FFLV patterns with bases of representations.
\begin{cor}
For every $T\in\Pi_\la$ choose a monomial $M_T\in\mU^A$ of the form $\prod (f_{i,j}^A)^{T_{i,j}}$ where the factors are ordered arbitrarily. The set $\{M_Tv_\la^A\mid T\in\Pi_\la\}$ constitutes a basis in $L_\la^A$.
\end{cor}

We are now ready to prove Theorem~\ref{main}.
\begin{proof}[Proof of Theorem~\ref{main}]
Recall the $\mathfrak h^*$-grading $\deg$ of $R_{\bf d}^A$ given by $\deg X_{i_1,\ldots,i_k}=\om_k$. 
We are to show that the $\deg$-homogeneous ideal $J$ spanned by all $\deg$-homogeneous polynomials $X\in R_{\bf d}^A$ 
vanishing on $N^A y_\la^A$ is precisely $I_{\bf d}^A$. However, the commutation relations~(\ref{Acomm}) easily imply that the Lie algebra 
$\mathfrak n_-^A$ is nilpotent which means that its exponential map to $N^A$ is surjective (in view of the simply connectedness, 
it is bijective).  Therefore, $N^A y_\la^A=\exp(\mathfrak n_-^A)y_\la^A$.

For any $f\in\mathfrak n_-^A$ the point
$\exp(f)y_\la^A\in\mathbb P(U_\la^A)$ corresponds to the line $$\mathbb C(\exp(f)v_{\om_1}^A)^{\otimes a_1}\otimes\ldots\otimes(\exp(f)v_{\om_{n-1}}^A)^{\otimes a_{n-1}}$$ 
and, by the definition of the Segre embedding, coincides with the point 
$$(\exp(f)u_{\om_1}^A)^{a_1}\times\ldots\times(\exp(f)u_{\om_{n-1}}^A)^{a_{n-1}}\in
\mathbb P(L_{\om_1}^A)^{a_1}\times\ldots\times\mathbb P(L_{\om_{n-1}}^A)^{a_{n-1}}$$ 
and with the point $$(\exp(f)u_{\om_1}^A)\times\ldots\times(\exp(f)u_{\om_{n-1}}^A)\in \mathbb P_{\bf d}^A.$$

Let us denote $C^A_{i_1,\ldots,i_k}(\{c_{i,j}\})$ the coordinate of $\exp(f)v_{\om_k}^A$ corresponding to the basis vector 
$e^A_{i_1,\ldots,i_k}$ where $f=\sum c_{i,j}f_{i,j}^A$. Via the standard expansion of the action of $\exp(f)$ on $L_{\om_k}$ as 
a power series in the action of $f$ one sees that each $C^A_{i_1,\ldots,i_k}(\{c_{i,j}\})$ depends polynomially on the $c_{i,j}$. 
We may, therefore, view $C^A_{i_1,\ldots,i_k}$ as an element of $\mathbb C[\{z_{i,j}\}]$ with 
$1\le i<j\le n$. Let us introduce $n-1$ additional indeterminates $z_1,\ldots,z_{n-1}$. Then $J$ is precisely the kernel of the 
homomorphism $\psi^A:R_{\bf d}^A\to \mathbb C[\{z_{i,j},z_l\}]$ mapping $X_{i_1,\ldots,i_k}^A$ to $z_k C^A_{i_1,\ldots,i_k}$. 

Next, for $f'=\sum_{i,j}c_{i,j}f_{i,j}\in\mathfrak n_-$ denote $C_{i_1,\ldots,i_k}(\{c_{i,j}\})$ the coordinate of 
$\exp(f')v_{\om_k}$ (for $\exp:\mathfrak n_-\to N$) corresponding to the basis vector $e_{i_1,\ldots,i_k}$. The same reasoning shows that 
$C_{i_1,\ldots,i_k}$ also depends polynomially on the $c_{i,j}$ and we view $C_{i_1,\ldots,i_k}$ as an element of 
$\mathbb C[\{z_{i,j}\}]$. One may then observe that $C^A_{i_1,\ldots,i_k}$ is the initial part of $C_{i_1,\ldots,i_k}$ with respect to the grading $\grad_z^A$ defined by $\grad_z^A(z_{i,j})=a_{i,j}$. Note that this initial part is the component of grading 
$s_{i_1,\ldots,i_k}^A$ in $C_{i_1,\ldots,i_k}$.

For any $\grad^A$-homogeneous $X\in I_{\bf d}^A$ we have a relation $X'\in I_{\bf d}$ such that $\initial_{\grad^A}(\varphi(X'))=X$. On one hand, in view of Theorem~\ref{classicalplucker}, $\psi(X')=0$ where $\psi$ is the homomorphism from $R_{\bf d}$ to $\mathbb C[\{z_{i,j},z_l\}]$ mapping $X_{i_1,\ldots,i_k}$ to $z_k C_{i_1,\ldots,i_k}$. On the other hand, $\psi^A(X)$ is the initial part (the component of grading 
$\grad^A(X)$) of $\psi(X')$ with respect to $\grad_z^A$ where $\grad_z^A$ is the extension of $\grad^A$ by $\grad_z^Az_l=0$, 
i.e. $\psi^A(X)=0$. This shows that $I_{\bf d}^A\subset J$.

The ring $\mathbb C[\{z_{i,j},z_l\}]$ is $\mathfrak h^*$-graded by $\deg z_k=\om_k$ and $\deg z_{i,j}=0$; the homomorphisms 
$\psi$ and $\psi^A$ are $\deg$-homogeneous. To prove the reverse inclusion $J\subset I_{\bf d}^A$ we show that the dimensions of the homogeneous components of the image of $\psi^A$ are no less than the corresponding dimensions for the ring $R_{\bf d}^A/I_{\bf d}^A$. 

For a weight $\mu\in \mathbb Z_{\ge 0}\{\om_{d_1},\ldots,\om_{d_s}\}$ the component of degree $\mu$ has dimension $\dim L_\mu$ which is simply due to $I_{\bf d}^A$ being an initial ideal of $I_{\bf d}$. We extend the notations from subsection~\ref{tableaux} by denoting $X^A(Z)$ the indeterminate $X_{i_1,\ldots,i_k}$ for the unique PBW tableau $Z$ of shape $\om_k$ and content $\{i_1,\ldots,i_k\}$. For a PBW tableau $Y$ of shape $\mu$ with columns $Z_1,\ldots,Z_{\mu_1}$ we denote $X^A(Y)=X^A(Z_1)\ldots X^A(Z_{\mu_1})$. We prove the announced inequality by showing that the images under $\psi^A$ of the monomials $X^A(Y)$ with $Y\in\mathcal Y_\mu$ are linearly independent.

Let us take a closer look at the polynomial $C^A_{i_1,\ldots,i_k}$: it is the sum of products $\pm z_{i_1,j_1}\ldots z_{i_N,j_N}/N!$ over all sequences $(i_1,j_1),\ldots,(i_N,j_N)$ such that $$\deg^A(f_{i_1,j_1}^A\ldots f_{i_N,j_N}^A)=s^A_{i_1,\ldots,i_k}$$ and $$(f_{i_1,j_1}^A\ldots f_{i_N,j_N}^A)v^A_{\om_k}=\pm e_{i_1,\ldots,i_k}^A.$$ In particular, we see that the monomial $\prod z_{i,j}^{T(e_{i_1,\ldots,i_k}^A)_{i,j}}$ appears with coefficient $(-1)^{\sigma_{i_1,\ldots,i_k}}$ for $\sigma_{i_1,\ldots,i_k}\in S_k$ such that $(i_{\sigma_{i_1,\ldots,i_k}(1)},\dots,i_{\sigma_{i_1,\ldots,i_k}(k)})$ is a PBW tableau. The partial order $\preceq$ from the proof of Proposition~\ref{fflvlinind} induces a monomial order on $\mathbb C[\{z_{i,j}\}]$ by setting $$z_{i_1,j_1}\ldots z_{i_N,j_N}\preceq z_{i'_1,j'_1}\ldots z_{i'_{N'},j'_{N'}}$$ whenever $$f_{i_1,j_1}^A\ldots f_{i_N,j_N}^A\preceq f_{i'_1,j'_1}^A\ldots f^A_{i'_{N'},j'_{N'}}.$$ The discussion in the proof of Proposition~\ref{fflvlinind} implies that $(-1)^{\sigma_{i_1,\ldots,i_k}}\prod z_{i,j}^{T(e_{i_1,\ldots,i_k}^A)_{i,j}}$ is the initial term in $C^A_{i_1,\ldots,i_k}$ with respect to monomial order $\preceq$.

We extend $\preceq$ to a multiplicative partial order (also $\preceq$) on the set of monomials in $\mathbb C[\{z_{i,j},z_l\}]$ 
by comparing the images of these monomials in $\mathbb C[\{z_{i,j}\}]$ under the map sending every $z_l$ to 1 and every $z_{i,j}$ 
to itself. If we denote $b_1,\ldots,b_{n-1}$ the coordinates of $\mu$ in the basis of fundamental weights, then we see that 
$\pm\prod z_l^{b_l}\prod z_{i,j}^{\tau(Y)_{i,j}}$ is the unique $\preceq$-minimal term in $\psi^A(X^A(Y))$. Since these monomials 
are distinct for distinct $Y$, the linear independence of the polynomials $X^A(Y)$ with $Y\in\mathcal Y_\mu$ follows.
\end{proof}

The above proof has the following implication.
\begin{cor}\label{Apbwssyt}
For $\mu\in \mathbb Z_{\ge 0}\{\om_{d_1},\ldots,\om_{d_s}\}$, the images of the monomials $X^A(Y)$ with $Y$ ranging over $\mathcal Y_\mu$ comprise a basis in the component of $R_{\bf d}^A/I_{\bf d}^A$ of homogeneity degree $\mu$.
\end{cor}

\section{Toric degenerations and monomial annihilating ideals}\label{toric}

The following fact may easily be deduced from the above results.
\begin{theorem}\label{toricthm}
If the weight system $A$ is such that all inequalities of types (a) and (b) (or, equivalently, all inequalities of types (A) and (B)) are strict, then the degenerate flag variety $F_\la^A$ is the toric variety corresponding to the polytope $Q_\la$.
\end{theorem}
\begin{proof}
The proof of Proposition~\ref{antichainmin} shows that the only monomial taking $v_{\om_k}^A$ to some $\pm e^A_{i_1,\ldots,i_k}$ is $\prod (f_{i,j}^A)^{T(e^A_{i_1,\ldots,i_k})_{i,j}}$ (up to a permutation of the commuting factors). The proof of Theorem~\ref{main} then shows that 
\begin{equation}\label{zfflvmon}
\psi^A(X_{i_1,\ldots,i_k})=(-1)^{\sigma_{i_1,\dots,i_k}}z_k\prod z_{i,j}^{T(e^A_{i_1,\ldots,i_k})_{i,j}}.
\end{equation}
When $k=d_j$, the monomials~(\ref{zfflvmon}) generate the coordinate ring of the toric variety given by the polytope $Q_{\om_k}$. However, the image of $\psi ^A$ (which is isomorphic to $R_{\bf d}^A/I_{\bf d}^A$) is generated by the monomials~(\ref{zfflvmon}) with $k$ ranging over $\{d_1,\ldots,d_s\}$ and all possible $1\le i_1<\ldots<i_k\le n$. The theorem now follows from Lemma~\ref{minkowski}.
\end{proof}

Now note that the commutation relations~(\ref{Acomm}) show that the Lie algebra $\mathfrak n_-^A$ and the algebra $\mU^A$ are commutative whenever all the inequalities of type (A) are strict. In other words, we have $\mU^A=\mathbb C[\{f_{i,j}^A\}]$
which allows us to speak of monomial ideals in $\mU^A$.
\begin{theorem}
If the weight system $A$ is such that all inequalities of types (a) and (b) are strict, then the annihilating ideal of $L_\la^A$ in $\mU^A$ is the monomial ideal spanned by the monomials $\prod (f_{i,j}^A)^{S_{i,j}}$ with $S$ ranging over the set $\mathbb Z_{\ge 0}^{\{1\le i<j\le n\}}\backslash\Pi_\la$.
\end{theorem}
\begin{proof}
We are to show that every monomial $\prod (f_{i,j}^A)^{S_{i,j}}$ with $S\in\mathbb Z_{\ge 0}^{\{1\le i<j\le n\}}\backslash\Pi_\la$ acts trivially on $L_\la^A$. We do so by invoking Lemma~\ref{tensor} and showing that it acts trivially on $w_\la^A\in U_\la^A$. Indeed,  $\prod (f_{i,j}^A)^{S_{i,j}}(w_\la^A)$ is equal to a sum of expressions of the form 
\begin{multline}\label{tensordecomp1}
\left(\prod (f_{i,j}^A)^{S^{1,1}_{i,j}}(v_{\om_1}^A)\right)\otimes\ldots\otimes\left(\prod (f_{i,j}^A)^{S^{1,a_1}_{i,j}}(v_{\om_1}^A)\right)\otimes\ldots\\\otimes\left(\prod (f_{i,j}^A)^{S^{n-1,1}_{i,j}}(v_{\om_{n-1}}^A)\right)\otimes\ldots\otimes\left(\prod (f_{i,j}^A)^{S^{n-1,a_{n-1}}_{i,j}}(v_{\om_{n-1}}^A)\right)
\end{multline}
where the sum of all $S_{i,j}^{k,l}$ is equal to $S_{i,j}$ for all $i,j$. However, every monomial in the $f_{i,j}^A$ takes $v_{\om_k}^A$ to a vector of the form $\pm e_{i_1,\ldots,i_k}^A$ and the  only monomial taking $v_{\om_k}^A$ to $\pm e^A_{i_1,\ldots,i_k}$ is $\prod (f_{i,j}^A)^{T(e^A_{i_1,\ldots,i_k})_{i,j}}$. Hence, every monomial $\prod (f_{i,j}^A)^{T_{i,j}}$ with $T\notin\Pi_{\om_k}$ annihilates $v_{\om_k}^A$. Therefore,~(\ref{tensordecomp1}) is nonzero only if every $\{S^{k,l}_{i,j}\}_{1\le i<j\le n}\in\Pi_{\om_k}$ which is impossible in view of 
$S\notin\Pi_\la$.
\end{proof}

We have shown that the same combinatorial conditions (which hold for a ``generic'' weight system) are sufficient for both $F_\la^A$ being the toric variety given by $Q_\la$ and for the annihilating ideal of $L_\la^A$ being spanned by the monomials with exponent vectors outside of $Q_\la$. This is not a coincidence, a slightly closer look at the proof of Theorem~\ref{main} shows that whenever the annihilating ideal of $L_\la^A$ is as above, $F_\la^A$ is the toric variety in question.

\begin{remark}
Recall that in \cite{FFR} the authors construct a weight system $A$ such that the annihilating ideal of a cyclic vector of
$L_\lambda^A$ is monomial. This weight system satisfies all the inequalities of types (A) and (B).  
\end{remark}

\section{The cone of degenerations}\label{cone}

Conditions (a) and (b) define a polyhedral cone $\mathcal K$ in $\mathbb R^{\{1\le i<j\le n\}}$. One sees that $\mathcal K$ is the product of the linear subspace of dimension $n-1$ given by $a_{i,j}=a_{i,i+1}+\ldots+a_{j-1,j}$ for all $1\le i<j\le n$ and of the simplicial cone of dimension ${n-1\choose 2}$ given by all $a_{i,i+1}=0$ and all inequalities of types (a) and (b). We show that, in a sense, all the degenerations that are obtained in this paper are parametrized by the faces of $\mathcal K$.

\begin{proposition}\label{facesame}
If the minimal face of $\mathcal K$ containing weight system $A$ coincides with the minimal face of $\mathcal K$ containing some weight system $B=\{b_{i,j}\}$, then the following hold.
\begin{enumerate}[label=\roman*)]
\item The map $\Theta:\mathfrak n_-^A\to \mathfrak n_-^B$ taking $f_{i,j}^A$ to $f_{i,j}^B$ is an isomorphism of Lie algebras.
\item The representations $L_\la^A$ and $L_\la^B$ are isomorphic (where $\mathfrak n_-^A$ and $\mathfrak n_-^B$ are identified by $\Theta$).
\item The varieties $F_\la^A$ and $F_\la^B$ coincide as subvarieties in $\mathbb P_{\bf d}^A$ (where $\mathbb P_{\bf d}^B$ is identified with $\mathbb P_{\bf d}^A$ via the isomorphisms between $L_{\om_k}^A$ and $L_{\om_k}^B$).
\end{enumerate}
\end{proposition}
\begin{proof}
\quad

i) The commutation relations~(\ref{Acomm}) are determined by which inequalities of type (A) are strict and which are not. As seen in the proof of Proposition~\ref{moreineqs}, every inequality of type (A) decomposes into the sum of inequalities of types (a) and (b). Hence, the commutation relations in $\mathfrak n_-^a$ are determined by which inequalities of types (a) and (b) are strict and which are not, i.e. by the minimal face of $\mathcal K$ containing $A$.

ii) First let us consider the case of a fundamental weight $\la=\om_k$. We show that the map sending $e_{i_1,\ldots,i_k}^A$ to $e_{i_1,\ldots,i_k}^B$ provides the desired isomorphism. For some $e_{i_1,\ldots,i_k}^A$
suppose that a monomial $M$ in the $f_{i,j}^A$ is such that $M v_{\om_k}^A=\pm e_{i_1,\ldots,i_k}^A$. Then, as seen in the proof of Proposition~\ref{antichainmin}, the monomial $\prod (f_{i,j}^A)^{T(e^A_{i_1,\ldots,i_k})_{i,j}}$ can be obtained from $M$ by replacing $f^A_{i,j}f^A_{j,l}$ with $f_{i,l}^A$ for $i<j<l$, replacing $f^A_{i,j}f^A_{l,m}$ with $f^A_{i,m}f^A_{l,j}$ for $i<l<j<m$ and commuting the factors. However, $\deg^A M=s^A_{i_1,\ldots,i_k}$ and none of these operations increase the degree $\deg^A$, therefore, each one of them must preserve $\deg^A$. 

In other words, whether a monomial takes $v_{\om_k}$ to $\pm e_{i_1,\ldots,i_k}^A$ depends on whether one may obtain $\prod (f_{i,j}^A)^{T(e^A_{i_1,\ldots,i_k})_{i,j}}$ from $M$ by a series of the above operations preserving $\deg^A$, i.e. on which inequalities of types (A) and (B) are equalities. The same holds for weight system $B$ and our assertion follows.

The general case now follows directly from Lemma~\ref{tensor}.

iii) The polynomial $C^A_{i_1,\ldots,i_k}$ introduced in the proof of Theorem~\ref{main} is determined by the set of monomials in $\mU^A$ taking $v_{\om_k}^A$ to $\pm e_{i_1,\ldots,i_k}^A$. This set, as observed in the proof of part ii), is determined by the minimal face of $\mathcal K$ containing $A$. Therefore, the map $\psi^A$ and its kernel $I_{\bf d}^A$ are also determined by the minimal face of $\mathcal K$ containing $A$ and our assertion follows.
\end{proof}

In particular, we see that $\mathfrak n_-^A$, $L_\la^A$ and $F_\la^A$ coincide with the non-degenerate objects 
$\mathfrak n_-$, $L_\la$ and $F_\la$ whenever $A$ is contained in the minimal face of $\mathcal K$, i.e. the $(n-1)$-dimensional linear space given by the equations $a_{i,j}=a_{i,i+1}+\ldots+a_{j-1,j}$. When all inequalities of type (a) (or (A)) are strict but all the inequalities of type (b) (or (B)) are equalities, $\mathfrak n_-^A$, $L_\la^A$ and $F_\la^A$ coincide with the abelian degenerations $\mathfrak n_-^a$, $L_\la^a$ and $F_\la^a$. When the minimal face containing $A$ is maximal, i.e. $A$ lies in the interior of $\mathcal K$, we are in the toric situation discussed in Section~\ref{toric}.

We have shown that the ideal of the variety $F_\la^A$ is an initial ideal of the ideal of the non-degenerate flag variety $F_\la$. We will now generalize this fact.
\begin{proposition}\label{ABinitial}
Suppose that the weight system $B$ is such that the minimal face of $\mathcal K$ containing $B$ contains the minimal face of $\mathcal K$ containing $A$. Then the ideal $(\varphi^B)^{-1}(I_{\bf d}^B)$ is the initial ideal of $(\varphi^A)^{-1}(I_{\bf d}^A)$ with respect to grading $\grad^B$ on $R_{\bf d}$.
\end{proposition}
\begin{proof}
Since the dimensions of the homogeneous components coincide, it suffices to consider some 
$\delta\in I_{\bf d}^A$ and show that $\initial_{\grad^B}((\varphi^A)^{-1}(\delta))\in(\varphi^B)^{-1}(I_{\bf d}^B)$. 
In other words, we are to show that $$\psi^B\left(\varphi^B(\initial_{\grad^B}((\varphi^A)^{-1}(\delta)))\right)=0.$$

Now, if a monomial $M$ in the $f_{i,j}^B$ takes $v_{\om_k}^B$ to $\pm e_{i_1,\ldots,i_k}^B$, then the monomial obtained from 
$M$ by replacing the $B$ superscripts with $A$ superscripts takes $v_{\om_k}^A$ to $\pm e_{i_1,\ldots,i_k}^A$. This follows from the set of inequalities of types (a) and (b) that are equalities being larger for $A$ than for $B$. This implies that $C^B_{i_1,\ldots,i_k}$ is the initial part of $C^A_{i_1,\ldots,i_k}$ with respect to $\grad_z^B$. Therefore, 
$$\psi^B\left(\varphi^B(\initial_{\grad^B}((\varphi^A)^{-1}(\delta)))\right)$$ is the initial part of $\psi^A(\delta)=0$ with respect to 
$\grad_z^B$.
\end{proof}

We may now replicate Proposition~\ref{free} and the preceding discussion to show that in the assumptions of Proposition~\ref{ABinitial} the variety $F_\la^B$ provides a flat degeneration of $F_\la^A$.

To complete this section we derive Proposition~\ref{Aquadratic}.
\begin{proof}[Proof of Proposition~\ref{Aquadratic}]
First we consider the case when $A$ lies in the interior of $\mathcal K$, i.e. the case discussed in Section~\ref{toric}. If we denote $J_{\bf d}^A$ the ideal in $R_{\bf d}^A$ generated by the quadratic part of $I_{\bf d}^A$, then, in view of Corollary~\ref{Apbwssyt}, it suffices to show that for every monomial $M\in R_{\bf d}^A$ there exists some $\delta\in J_{\bf d}^A$ such that $M+\delta$ is a linear combination of monomials of the form $X^A(Y)$ with $Y$ PBW semistandard.

Recall the partial order $\preceq$ from Remark~\ref{lattice}. Consider a product $X^A_{i_1,\ldots,i_k}X^A_{i'_1,\ldots,i'_{k'}}\in R_{\bf d}^A$ such that $k\ge k'$ and $\{i_1,\ldots,i_k\}\not\preceq\{i'_1,\ldots,i'_{k'}\}$. Let $Z$ denote the PBW tableau with elements $i_1,\ldots,i_k$ in its first column and $i'_1,\ldots,i'_{k'}$ in its second column. By our assumption, $Z$ in not PBW semistandard, i.e. we have a maximal 
$i_0$ such that for all $i_0\le i\le k$ we have $Z_{i,1}<Z_{i_0,2}$. Let us denote $W$ the tableau obtained from $Z$ by exchanging $Z_{i,1}$ with $Z_{i,2}$ for all $i\le i_0$. Note that $W$ is also a PBW tableau and that $W_{i_0,2}<Z_{i_0,2}$. Let $j_1<\ldots<j_k$ be the elements in the first column of $W$ and $j'_1<\ldots<j'_{k'}$ be the elements in its second column. Then we have 
$$T(e_{i_1,\ldots,i_k}^A)+T(e_{i'_1,\ldots,i'_{k'}}^A)=T(e_{j_1,\ldots,j_k}^A)+T(e_{j'_1,\ldots,j'_{k'}}^A)$$ 
which, in view of the discussion in Section~\ref{toric}, means that we have a relation of the form $$X^A_{i_1,\ldots,i_k}X^A_{i'_1,\ldots,i'_{k'}}\pm X^A_{j_1,\ldots,j_k}X^A_{j'_1,\ldots,j'_{k'}}\in I_{\bf d}^A.$$

Now consider any monomial $X^A(Y)\in R_{\bf d}^A$ with $Y$ a PBW tableau that is not PBW semistandard. For some column $j$ we have $\{Y_{\cdot,j}\}\not\preceq\{Y_{\cdot,j+1}\}$ and we may apply the procedure from the previous paragraph to the $j$th and $j+1$st column of $Y$ to obtain a new PBW tableau $Y'$ with $X^A(Y)\pm X^A(Y')\in J_{\bf d}^A$. Furthermore, the lowest differing element in the rightmost differing column of $Y$ and $Y'$ will be smaller in $Y'$ than in $Y$, i.e. $Y'$ will be smaller with respect to the corresponding lexicographic order. Consequently, iteration of this procedure will provide the desired relation for the monomial $X^A(Y)$.  

We derive the general case from the toric case. Consider a relation $\delta\in I_{\bf d}^A$, homogeneous with respect to both 
$\deg$ and $\grad^A$. Consider any weight system $B$ in the interior of $\mathcal K$. By Proposition~\ref{ABinitial} we have an 
$\delta'\in I_{\bf d}^B$ such that $$\initial_{\grad^B}((\varphi^A)^{-1}(\delta))=(\varphi^B)^{-1}(\delta').$$ Let 
$\delta'=K_1 P'_1+\ldots+ K_n P'_n$ for $\deg$-homogeneous $K_i\in R_{\bf d}^B$ and quadratic $\deg$-homogeneous and 
$\grad^B$-homogeneous $P'_i\in I_{\bf_d}^B$. We have quadratic $P_i\in I_{\bf d}^A$ with 
$$\varphi^B(\initial_{\grad^B}((\varphi^A)^{-1}(P_i)))=P'_i$$ and, consequently, 
\begin{multline*}
\initial_{\grad^B}((\varphi^B)^{-1}(K_1) (\varphi^A)^{-1}(P_1)+\ldots+(\varphi^B)^{-1}(K_n)(\varphi^A)^{-1}(P_n))=\\\initial_{\grad^B}((\varphi^B)^{-1}(\delta')).
\end{multline*}
Therefore, if we introduce the grading $\grad^B$ on $R_{\bf d}^A$ via $\varphi^A$, the $\grad^B$ grading of the $\grad^B$-initial part of the difference 
$$\delta-\varphi^A((\varphi^B)^{-1}(K_1))P_1-\ldots-\varphi^A((\varphi^B)^{-1}(K_n))P_n$$ is greater than $\grad^B(\delta')$. 
However, within a given $\deg$-homogeneous component the grading $\grad^B$ takes only a finite number of values. Therefore, passing from 
$\delta$ to the above difference and then iterating the procedure will express $\delta$ as a $R_{\bf d}^A$-linear combination of 
polynomials in the quadratic part of $I_{\bf d}^A$.
\end{proof} 

\begin{remark}\label{pbwlocus}
Let us consider the face $\mathcal{F}$ of $\mathcal K$ defined by the condition that all inequalities in (B) are equalities.
$\mathcal{F}$ has $2^{n-2}$ subfaces determined by the set of inequalities in (A) turning into equalities. There is a one-to-one correspondence between the subfaces of $\mathcal{F}$ and the weight systems given in Proposition 7 of \cite{CFFFR}.
More precisely, each subface contains exactly one weight system from Proposition 7 of \cite{CFFFR} (the number of these weight systems
is $2^{n-2}$ -- note the notation shift $n+1\to n$ from \cite{CFFFR} to this paper).
\end{remark}

\section{A maximal cone in the tropical flag variety}\label{tropical}
We briefly recall some basic facts on the tropical flag varieties (see \cite{MaS}, \cite{SS}).
Let us consider a $(2^n-2)$-dimensional real vector space. A point ${\bs}\in \bR^{2^n-2}$ has coordinates 
$s_I=s_{i_1,\dots,i_k}$ labeled by collections $1\le i_1<\dots <i_k\le n$ with $1\le k<n$.
A point $\bs$ defines a grading on the polynomial ring in Pl\"ucker variables $X_I$ attaching degree $s_I$ to $X_I$.  

We consider the Pl\"ucker embedding of the flag variety $\mathcal{F}l_n:=SL_n/B$ into the product of projective spaces.
For the ideal $J$ consisting of all multi-homogeneous polynomials in $X_I$ vanishing on the image of $SL_n/B$
and a point ${\bs}\in \bR^{2^n-2}$ we consider the initial ideal $J_{\bs}$ (with respect to the grading defined by $\bs$).
The tropical flag variety $\trop(\mathcal{F}l_n)$ consists of points $\bs$ such that $J_{\bs}$ does not contain monomials. 

\begin{remark}\label{zerocond}
From the tropical point of view, it is natural to quotient the space $\bR^{2^n-2}$ by $\bR^{n-1}$, since a point 
${\bs}=(s_I)$ belongs
to the tropical flag variety if and only if so do all the points of the form $\{s_I+b_{|I|}\}$ for any collection of numbers 
$(b_k)_{k=1}^{n-1}\in\bR^{n-1}$. In what follows we choose a normalization assuming that 
$s_{1,\dots,k}=0$ for all $k=1,\dots,n-1$.
\end{remark}

We consider the map from the cone 
$\mathcal K$ of weight systems to $\bR^{2^n-2}$ that maps $A=\{a_{i,j}\}$ to $\{s^A_I\}$, $I=\{i_1,\dots,i_k\}$ defined in Proposition \ref{antichainmin}. Recall that  
if $p_1<\ldots<p_l$ comprise the difference $\{1,\ldots,k\}\backslash\{i_1,\ldots,i_k\}$ and the integers $q_1>\ldots>q_l$ 
comprise the difference $\{i_1,\ldots,i_k\}\backslash\{1,\ldots,k\}$, then 
\begin{equation}\label{sA}
s^A_{i_1,\ldots,i_k}=a_{p_1,q_1}+\ldots+a_{p_l,q_l}.
\end{equation}
Let $h:\bR^{n(n-1)/2}\to \bR^{2^n-2}$ be the linear map defined by formula \ref{sA}.
We denote the image $h(\mathcal K)$ by $\mathcal C$. In particular, $\mathcal C$ is contained in the 
$(2^n-2)$-dimensional vector space,
whose coordinates will be denoted by $s_I$, $I$ being a proper subset of $\{1,\dots,n\}$.
We note that for any point $\bs$ in the image of $h$ one has $s_{1,\dots,k}=0$ for $1\le k\le n-1$.
The following Lemma describes $\mathcal C$ explicitly. 
\begin{lemma}
The cone $\mathcal C$ is cut out by the following set of equalities and inequalities: 
\begin{enumerate}[label=\lbrack\roman*\rbrack]
\item\label{1} $s_{1,\dots,k}=0$ for $1\le k\le n-1$,
\item\label{2} for any pair $1\le i<j\le n$ and any $i\le k<l<j$ one has $s_{1,\dots,i-1,i+1,\dots,k,j}=s_{1,\dots,i-1,i+1,\dots,l,j}$,
\item\label{3} given a proper $I$, let $p_s,q_s$, $s=1,\dots,l$ be the numbers defined above; then 
$s_I=s_{1,\dots,p_1-1,q_1}+\dots + s_{1,\dots,p_l-1,q_l}$;
\item\label{4} $s_{1,\dots,i-1,i+1}+s_{1,\dots,i,i+2}\ge s_{1,\dots,i-1,i+2}$ for $1\le i\le n-2$,
\item\label{5} $s_{1,\dots,i-1,j}+s_{1,\dots,i,j+1}\ge s_{1,\dots,i-1,j+1}+s_{1,\dots,i,j}$ for $1\le i<j-1\le n-2$.
\end{enumerate}
The image $h(\bR^{n(n-1)/2})$ is cut out by the first three linear relations.
\end{lemma}
\begin{proof}
The first and third conditions are obvious from the definition. The second condition comes from the fact that for any $i<k<j$ one has 
$s^A_{1,\dots,i-1,i+1,\dots,k,j}=a_{i,j}$. Fourth and fifth conditions come from the definition of a weight system.
\end{proof}

\begin{theorem}\label{Trop}
$\mathcal C$ is a maximal cone in the tropical flag variety for any $n>1$. 
\end{theorem}
\begin{proof}
Since all the degenerate flag varieties are irreducible, the cone $\mathcal C$ is contained in the tropical flag variety $\trop(\mathcal{F}l_n)$.
Now assume that there exists a larger cone in $\trop(\mathcal{F}l_n)$ containing $\mathcal C$. Since the dimension of a maximal cone in $\trop(\mathcal{F}l_n)$
is equal to $n(n-1)/2=\dim SL_n/B$ (see e.g. \cite{SS}) and $\dim \mathcal C=n(n-1)/2$ we conclude that if a maximal cone 
$\mathcal C'\subset \trop(\mathcal{F}l_n)$ contains $\mathcal C$, then $\mathcal C'$ is contained in $h(\bR^{n(n-1)/2})$
(if a point of $\mathcal C'$ is not contained in $h(\bR^{n(n-1)/2})$, then the dimension of the convex hull 
of this point and $\mathcal C$ is greater than $\dim \mathcal C$). This implies that if a point $\bs$ is contained in the
maximal cone $\mathcal C'\subset \trop(\mathcal{F}l_n)$, then $\bs$ satisfies conditions \ref{1}, \ref{2}, \ref{3}. 
So we are left to show that a point $\bs\in\mathcal C'$ satisfies conditions \ref{4} and \ref{5}. 

For $1\le i\le n-2$ we consider Pl\"ucker relation 
\[
X_{1,\dots,i,i+2}X_{1,\dots,i-1,i+1} - X_{1,\dots,i,i+1}X_{1,\dots,i-1,i+2} - X_{1,\dots,i-1,i+1,i+2}X_{1,\dots,i}. 
\]
If $s_{1,\dots,i-1,i+1}+s_{1,\dots,i,i+2}< s_{1,\dots,i-1,i+2}$, then the initial part of this relation
is a monomial $X_{1,\dots,i,i+2}X_{1,\dots,i-1,i+1}$, which implies that ${\bs}\notin \trop(\mathcal{F}l_n)$. We conclude that
inequality \ref{4} holds true for all $i$.

Now for $1\le i<j-1\le n-2$ we consider the Pl\"ucker relation
\[
X_{1,\dots,i-1,i+1,j}X_{1,\dots,i,j+1} - X_{1,\dots,i-1,i+1,j+1}X_{1,\dots,i,j} - X_{1,\dots,i-1,j,j+1}X_{1,\dots,i,i+1} 
\]
If $s_{1,\dots,i-1,j}+s_{1,\dots,i,j+1}< s_{1,\dots,i-1,j+1}+s_{1,\dots,i,j}$, then 
the initial term of this relation is the monomial $X_{1,\dots,i-1,i+1,j}X_{1,\dots,i,j+1}$, 
which implies that ${\bs}\notin \trop(\mathcal{F}l_n)$. We conclude that \ref{5} holds true. 

We have shown that for any $\bs\in\mathcal C'$ one has $\bs\in\mathcal C$. Hence $\mathcal C=\mathcal C'$.
\end{proof}

\section{Line bundles and BW theorem}\label{BW}

Consider the embedding $\iota_\la:F_\la^A\hookrightarrow\mathbb P(L_\la^A)$ and the line bundle $\mathcal O_{\mathbb P(L_\la^A)}(1)$. 
We denote $\mathcal L_\la^A$ the pullback $\iota_\la^*(\mathcal O_{\mathbb P(L_\la^A)}(1))$. Note that, since $F_\la^A$ depends 
only on $\bf d$, the sheaf $\mathcal L_\mu^A$ on $F_\la^A$ is now defined for all 
$\mu\in\mathbb Z_{\ge 0}\{\omega_{d_1},\ldots,\omega_{d_s}\}$.

These line bundles can also be obtained in a different way. The $\mathfrak h^*$-grading on $R_{\bf d}^A$, the coordinate 
ring of $\mathbb P_{\bf d}^A$, induces an $\mathfrak h^*$-grading on its structure sheaf. Denote $\mathcal O_{\bf d}(\la)$ the component 
of homogeneity degree $\la$. This sheaf is seen to be the restriction of $\mathcal O_{\mathbb P(U_\la^A)}(1)$. Therefore, the 
restriction of $\mathcal O_{\bf d}(\la)$ to $F_\la^A$ coincides with the restriction of $\mathcal O_{\mathbb P(U_\la^A)}(1)$ 
to $F_\la^A$, the latter being $\mathcal L_\la^A$.

In view of the previous paragraph, for $\mu\in\mathbb Z_{\ge 0}\{\omega_{d_1},\ldots,\omega_{d_s}\}$ the space of global sections 
$H^0(F_\la^A,\mathcal L_\mu^A)$ is the homogeneous component of degree $\mu$ in the coordinate ring $R_{\bf d}^A/I_{\bf d}^A$. 
Furthermore, $F_\la^A$ is obviously acted on by $N^A$ and $\mathcal L_\la^A$ is $N^A$-equivariant. This leads us to an 
analogue of the Borel-Weil theorem.
\begin{theorem}
As a representation of $N^A$ the space $H^0(F_\la^A,\mathcal L_\mu^A)$ is isomorphic to the dual $(L_\mu^A)^*$.  
\end{theorem} 
\begin{proof}
Since $\mathcal L_\mu^A$ is defined as the restriction of $\mathcal O_{\mathbb P(L_\mu^A)}(1)$ every functional on 
$L_\mu^A$ provides a global section of $\mathcal L_\mu^A$ and we have a map from $(L_\mu^A)^*$ to $H^0(F_\la^A,\mathcal L_\mu^A)$. 
This map is injective since the linear hull $\mathbb CN^A(v_\mu^A)=\mU^A v_\mu^A$ is all of $L_\mu^A$, i.e. $N^A(v_\mu^A)$ 
is not contained in any proper linear subspace, hence, no global section of $\mathcal O_{\mathbb P(L_\mu^A)}(1)$ vanishes on 
$F_\la^A$. However, we have identified $H^0(F_\la^A,\mathcal L_\mu^A)$ with the component of homogeneity degree $\mu$ in 
$R_{\bf d}^A/I_{\bf d}^A$ and Corollary~\ref{Apbwssyt} shows that the dimension of the latter component is precisely 
$\dim L_\mu^A$. The theorem follows.
\end{proof}

To extend this fact to an analogue of the Borel-Weil theorem (or, rather, its restriction to the case of integral dominant weights) 
we prove the acyclicity of $\mathcal L_\mu^A$.
\begin{theorem}
One has $H^m(F_\la^A,\mathcal L_\mu^A)=0$ for all $\mu\in\mathbb Z_{\ge 0}\{\omega_{d_1},\ldots,\omega_{d_s}\}$ and all $m>0$.
\end{theorem}
\begin{proof}
We first note that for a weight system $A$ lying in the interior of $\mathcal K$ the claim of the theorem holds
(since the line bundle $\mathcal L_\mu^A$ is generated by its sections, see \cite{fultor}, Section 3.5).
Now, given a weight system $B\in \mathcal K$, Proposition \ref{ABinitial} gives a flat family over $\mathbb A^1$
with the generic fiber $F_\la^B$ and $F_\la^A$ as the special fiber. Then the upper semi-continuity theorem 
for the cohomology groups \cite{H}, chapter III, Theorem 12.8 implies the claim of our theorem.
\end{proof}

\section*{Acknowledgments}
Igor Makhlin would like to thank the Max Planck Institute for Mathematics for providing a welcoming atmosphere, which stimulated the work on these subjects.
The work was partially supported by the grant RSF-DFG 16-41-01013.

\end{document}